\documentclass{amsart}

\usepackage{amsmath, amssymb, amsthm}
\usepackage{hyperref}
\usepackage[a4paper, bottom=2cm]{geometry}
\usepackage{tikz}
\usepackage{float}
\usepackage{multirow}

\usepackage{enumitem}
\usepackage{xcolor}

\DeclareMathOperator{\Hom}{Hom}
\DeclareMathOperator{\Aut}{Aut}
\DeclareMathOperator{\Inn}{Inn}
\DeclareMathOperator{\Out}{Out}
\DeclareMathOperator{\Syl}{Syl}
\DeclareMathOperator{\Sym}{Sym}

\DeclareMathOperator{\Alt}{Alt}

\DeclareMathOperator{\Sol}{Sol}

\renewcommand{\L}{\operatorname{L}}
\renewcommand{\O}{\operatorname{O}}

\DeclareMathOperator{\U}{U}

\DeclareMathOperator{\Spin}{Spin}

\DeclareMathOperator{\GL}{GL}
\DeclareMathOperator{\SL}{SL}
\DeclareMathOperator{\GO}{GO}
\DeclareMathOperator{\SO}{SO}
\DeclareMathOperator{\Sp}{Sp}

\DeclareMathOperator{\SU}{SU}

\DeclareMathOperator{\SDih}{SDih}

\DeclareMathOperator{\GF}{GF}

\DeclareMathOperator{\Suz}{Suz}
\DeclareMathOperator{\Fi}{Fi}
\DeclareMathOperator{\Co}{Co}
\DeclareMathOperator{\HN}{HN}
\DeclareMathOperator{\HS}{HS}
\newcommand{\ON}{\mathrm{O}'\mathrm{N}}
\DeclareMathOperator{\Ru}{Ru}
\DeclareMathOperator{\Th}{Th}
\DeclareMathOperator{\He}{He}
\DeclareMathOperator{\Bm}{B}
\DeclareMathOperator{\Jn}{J}
\DeclareMathOperator{\Mn}{M}
\DeclareMathOperator{\Mt}{M}
\DeclareMathOperator{\Ly}{Ly}
\DeclareMathOperator{\McL}{McL}

\newcommand{\TF}{{^2 \mathrm{F}_4(2)'}}

\newcommand{\bfU}{\mathbf{U}}
\newcommand{\bfT}{\mathbf{T}}

\newcommand{\bfW}{\mathbf{W}}
\newcommand{\bfV}{\mathbf{V}}

\newcommand{\Z}{\mathbb{Z}}

\newcommand{\calE}{\mathcal{E}}
\newcommand{\calF}{\mathcal{F}}
\newcommand{\calG}{\mathcal{G}}
\newcommand{\calH}{\mathcal{H}}

\newcommand{\normalIn}{\trianglelefteq}

\theoremstyle{definition}

\newtheorem{definition}{Definition}[section]
\newtheorem{notation}[definition]{Notation}

\newtheorem{algorithm}[definition]{Algorithm}

\theoremstyle{plain}
\newtheorem{theorem}[definition]{Theorem}
\newtheorem{lemma}[definition]{Lemma}
\newtheorem{proposition}[definition]{Proposition}
\newtheorem{corollary}[definition]{Corollary}

\title{Fusion Systems on Sylow $3$-subgroups of Fischer and Monster sporadic groups -- II}
\author{Pete Gautam}
\address{Department of Mathematics, University of Manchester, Manchester, M13 9PL, United Kingdom}
\email{pratyush.gautam@manchester.ac.uk}


\begin{document}
    
    \begin{abstract}
        We classify all corefree fusion systems on a Sylow 3-subgroup of the sporadic groups $\Fi_{24}'$ and $\Mn$. The program to classify  of all corefree fusion systems on Sylow $p$-subgroups of sporadic groups has been running for many years, comprising of work done in several papers by numerous authors. We complete this program for $p$ odd. In the process, we expand the theory of proto-essential subgroups and describe algorithms that make the process of finding all corefree fusion systems on a $p$-group better by several orders of magnitude.
    \end{abstract}
    
    \maketitle

    \section{Introduction}
    For a finite group $G$ and a $p$-subgroup $S$ of $G$, the \emph{fusion category} $\calF_S(G)$ is a category on the subgroups of $S$ with morphisms the conjugation maps coming from $G$. A \emph{fusion system} $\calF$ is a generalisation of fusion categories defined solely in terms of the $p$-group $S$. We are particularly interested in \emph{saturated fusion systems}, which include all \emph{realizable} fusion systems $\calF_S(G)$, where $S$ is further assumed to be a Sylow $p$-subgroup of $G$. Much interest is in the case when a saturated fusion system is not realized by any finite group $G$, called \emph{exotic} fusion systems (for a discussion, see \cite[III.7.4]{ako}).

    The classification of all `interesting' fusion systems on Sylow $p$-subgroups of sporadic groups has been an ongoing project for the last $20$ years, dating back to the initial work by Ruiz and Viruel in \cite{rv}. This project has resulted in the discovery of many exotic fusion systems. To complete the classification for $p$ odd, it remained to consider Sylow $3$-subgroups of the Fischer groups $\Fi_{22}$, $\Fi_{23}$, $\Fi_{24}'$, the Baby Monster group $\Bm$ and the Monster group $\Mn$. In \cite{paper:paper-1}, the author classifies all \emph{corefree} fusion systems $\calF$ on Sylow $3$-subgroup of $\Fi_{22}$, $\Fi_{23}$ and $\Bm$, which are precisely the fusion systems satisfying $O_3(\calF) = 1$. In this paper, we determine all corefree fusion systems on a Sylow $3$-subgroup of $\Fi_{24}'$ and $\Mn$, thus completing the project.
    \begin{theorem} \label{thm-1}
        Let $\calF$ be a saturated fusion system on a Sylow $3$-subgroup of $\Fi_{24}'$. If $O_3(\calF) = 1$, then $\calF$ is realized by $\Fi_{24}'$ or $\Fi_{24}$.
    \end{theorem}
    
    \begin{theorem} \label{thm-2}
        Let $\calF$ be a saturated fusion system on a Sylow $3$-subgroup of $\Mn$. If $O_3(\calF) = 1$, then $\calF$ is realized by $\Mn$.
    \end{theorem}
    \noindent Theorem \ref{thm-1} is proven as Theorem \ref{thm:fi24}, while Theorem \ref{thm-2} is proven as Theorem \ref{thm:m}.

    Although the completion of this classification has been a natural, and an important, objective, the sizes of the remaining $3$-groups had been a major obstacle to achieve this. Indeed, it was inaccessible to do so using computational methods that had been used previously. The most efficient algorithm to do so was defined by Parker and Semeraro in \cite{algorithms}. The process they describe first computes all subgroups of $S/Z(S)$, which is not tractable for such large groups. The aim of the process is to find all the subgroups of $S$ that are potentially essential in some saturated fusion system $\calF$ on $S$, the so-called \emph{proto-essential subgroups}. In practice, the number of proto-essential subgroups for a given $p$-group is in single digits (see, for example, \cite[Appendix A]{algorithms}). Moreover, the process of checking whether a subgroup is proto-essential is rather expensive, especially for large subgroups. As such, one would like to run this test on as few subgroups as possible.

    In this paper, we expand the theory of proto-essential subgroups. This allows us to define a new algorithm that computes all proto-essential subgroups of $S$ without requiring all the subgroups of $S/Z(S)$. This, in turn, allows us to find all proto-essential subgroups for relatively large groups by only considering an incredibly few number of subgroups. For instance, using the new algorithm, we consider around $160$ subgroups to find all the proto-essential subgroups of a Sylow $3$-subgroup of $\Fi_{23}$ (of order $3^{13}$). In contrast, we tested about $180 \ 000$ subgroups of $S$ using the GAP implementation of the Parker-Semeraro algorithm in \cite{paper:paper-1}. Clearly, the new algorithm is a major improvement to the status quo, and allows us to tackle problems in fusion systems that were previously intractable.

    Further to this algorithm, we present a result that allows us to infer that a given $p$-group does not support corefree fusion systems by considering an even fewer number of subgroups. For groups of order up to $p^6$, this typically involves checking whether less than $5$ subgroups of $S$ can be proto-essential (see Appendix \ref{sec:algorithm-appendix}), making it incredibly fast. Efficient methods at $p=2$ were described in \cite[Proposition 2.2]{aov} to determine whether a group supports `interesting' fusion systems, but there was nothing equivalent for $p$ odd. This algorithm provides us with tests for every prime $p$. 
    
    Using the tests in \cite[Proposition 2.2]{aov}, Andersen, Oliver and Ventura were able to show that every fusion system on a group of order up to $2^9$ is realizable in \cite[Theorem B]{aov}. We know that there is an exotic fusion system on a group of order $2^{10}$, namely the \emph{Benson-Solomon} fusion system $\calF_{\Sol}(3)$. It is widely believed that this is the only exotic fusion system on groups of order $2^{10}$. However, this problem has remained open for many years due to the sheer number of groups we would need to check to confirm this result. We believe that this problem is now tractable by combining the tests in \cite[Proposition 2.2]{aov} with the algorithm we describe here.
    
    Sporadic groups have paved the way in finding exotic fusion systems. Indeed, the aforementioned Benson-Solomon fusion systems were discovered by Solomon in \cite{sol} long before fusion systems were defined. Solomon constructed these groups by studying the $2$-local structure of $\Co_3$ and $\Spin_7(3)$. In particular, Solomon showed that a Sylow $2$-subgroup of $\Spin_7(q)$ (for $q$ odd) cannot be a Sylow $2$-subgroup of any finite group where all the involutions are conjugate. In \cite{lo-spin7q}, Levi and Oliver made use of this construction to define the Benson-Solomon fusion systems $\calF_{\Sol}(q)$ on a Sylow $2$-subgroup of $\Spin_7(q)$ where all the involutions are conjugate (see also \cite{annals-paper} for a group-theoretic construction). To date, these are the only known simple exotic fusion systems at $p = 2$.
    
    For $p$ odd, one of the earliest works in fusion theory was to classify all fusion systems on the extraspecial group $p^{1+2}_+$ by Ruiz and Viruel in \cite{rv}. In this monumental paper, they found that a corefree fusion system may be realized by some almost simple group $G$, where $O^{p'}(G) = F^*(G)$ is isomorphic to $\L_3(p)$ or a sporadic group. Furthermore, they found that the group $7^{1+2}_+$ supports three exotic fusion systems. These were one of the earliest examples of exotic fusion systems at $p$ odd.

    The work by Oliver and his collaborators led to the discovery of further exotic fusion systems on $p$-groups that have an abelian subgroup of index $p$ in \cite{abelian-1, abelian-2, abelian-3}. This includes some related to sporadic groups, such as a Sylow $5$-subgroup of $\Co_1$. Similarly, the work by Parker and his collaborators in \cite{pp-1, ps18} found exotic fusion systems on $p$-groups of order $p^{p-1}$ for $p \geq 7$, including the Sylow $7$-subgroup of $\Mn$, where the $p$-group has an extraspecial group of index $p$. More notably, in \cite{thesis:mm}, Moragues Moncho showed that these are the only fusion systems on a $p$-group with an extraspecial subgroup of index $p$. This result, combined with the previous two paragraphs, suggests that the presence of certain exotic fusion systems is deeply intertwined with the existence of sporadic groups.

    In \cite{pearls}, Grazian found an exotic fusion system on a $7$-subgroup of $\Mn$ of order $7^5$. This group is the smallest $p$-group that supports exotic fusion systems and does not have an abelian or an extraspecial subgroup of index $p$. This $7$-subgroup of $\Mn$, as well as the Sylow $7$-subgroup, are groups of maximal class. Following this work, Parker and Grazian have found many exotic fusion systems on $p$-groups of maximal class in \cite{maximal-class}.

    Although the sporadic groups mentioned above have led to the discovery of many exotic fusion systems, this is not the case for all of them. The work of Parker and Semeraro in \cite{algorithms} classified corefree fusion systems on small $p$-groups. Rather surprisingly, they did not find any further exotic fusion systems, including on those related to sporadic groups. This suggests that the structure of exotic fusion systems is very rigid, and might be closely related to \emph{pearls} (as defined by Grazian in \cite{pearls}), which have been the largest source of exotic fusion systems. Indeed, in \cite{vb-sporadics}, van Beek finds the remaining exotic fusion systems related to sporadic groups, and remarks that the exotic fusion systems he describes are closely related to fusion systems containing pearls.

    Up to that point, the remaining $p$-groups to consider were the Sylow $3$-subgroups of $\Fi_{22}$, $\Fi_{23}$, $\Fi_{24}'$, $\Bm$ and $\Mn$. In \cite{paper:paper-1}, the author shows that the Sylow $3$-subgroups of $\Fi_{22}$, $\Fi_{23}$ and $\Bm$ do not support exotic fusion systems. Indeed, as indicated by the main theorems, we find that the remaining two groups do not support exotic fusion systems. The Sylow $3$-subgroup of $\Fi_{24}'$ and $\Mn$ have such a rich structure that the only simple fusion systems are the ones realized by $\Fi_{24}'$ and $\Mn$. This is in contrast with the result of the author on a Sylow $3$-subgroup of $\Fi_{22}$ in \cite{paper:paper-1}.
    
    The following table, adapted from \cite[Table 1]{vb-sporadics}, summarizes the number of corefree, simple and exotic fusion systems supported on a Sylow $p$-subgroup of a sporadic group.
    \begin{table}[H]
        \centering
        \resizebox{\textwidth}{!}{\begin{tabular}{c|c|c|c|c|c}
            Sporadic Group & $|S|$ & Corefree & Simple & Exotic & Reference(s) \\
            \hline
            $\TF$, $\Jn_2$, $\Jn_4$, $\Mt_{12}$, $\Mt_{24}$, $\Ru$, $\He$ & $3^3$ & $4$ & $2$ & $0$ & \cite[Table 1.2]{rv} \\
            $\Jn_3$ & $3^5$ & $0$ & $0$ & $0$ & \cite{algorithms} \\
            $\HN$ & $3^6$ & $10$ & $3$ & $0$ & \cite[A.5.7]{algorithms} \\
            $\Co_2$, $\McL$ & $3^6$ & $13$ & $4$ & $0$ & \cite{mcl}, \cite[A.5.8]{algorithms} \\
            $\Suz$, $\Ly$ & $3^7$ & $3$ & $2$ & $0$ & \cite[A.6.3]{algorithms} \\
            $\Co_3$ & $3^7$ & $1$ & $1$ & $0$ & \cite[A.6.8]{algorithms} \\
            $\Fi_{22}$ & $3^9$ & $6$ & $3$ & $0$ & \cite[Section 5]{paper:paper-1} \\
            $\Co_1$ & $3^9$ & $3$ & $2$ & $0$ &  \cite{todd-modules}, \cite[Section 4]{vb-sporadics} \\
            $\Th$ & $3^{10}$ & $3$ & $3$ & $2$ & \cite[Section 5]{vb-sporadics} \\
            $\Fi_{23}$, $\Bm$ & $3^{13}$ & $6$ & $2$ & $0$ & \cite[Section 7]{paper:paper-1} \\
            $\Fi_{24}'$ & $3^{16}$ & $2$ & $1$ & $0$ & Section \ref{sec:fi24} \\
            $\Mn$ & $3^{20}$ & $1$ & $1$ & $0$ & Section \ref{sec:m} \\
            \hline
            $\Co_2$, $\Co_3$, $\Th$, $\HS$, $\McL$, $\Ru$ & $5^3$ & $3$ & $2$ & $0$ & \cite[Table 1.2]{rv} \\
            $\Co_1$ & $5^4$ & $30$ & $13$ & $24$ & \cite{abelian-1}, \cite{abelian-2}, \cite[A.3.4]{algorithms} \\
            $\HN$, $\Ly$, $\Bm$ & $5^6$ & $5$ & $4$ & $0$ & \cite[A.5.26]{ps18} \\
            $\Mn$ & $5^9$ & $4$ & $3$ & $3$ & \cite[Section 6]{vb-sporadics} \\
            \hline
            $\Fi_{24}'$, $\He$, $\ON$ & $7^3$ & $13$ & $6$ & $3$ & \cite[Table 1.2]{rv} \\
            $\Mn$ & $7^6$ & $26$ & $21$ & $27$ & \cite{pp-1}, \cite[Table 1]{ps18}  \\
            \hline
            $\Jn_4$ & $11^3$ & $2$ & $1$ & $0$ & \cite[Table 1.2]{rv} \\
            \hline
            $\Mn$ & $13^3$ & $5$ & $2$ & $0$ & \cite[Table 1.2]{rv}
        \end{tabular}}
        \caption{Corefree, simple and exotic fusion systems supported by a Sylow $p$-subgroup $S$ of a sporadic group ($p$ odd).}
    \end{table}

    We now briefly go through the techniques used in the paper. Starting with a Sylow $3$-subgroup, we construct local subsystems corresponding to the $3$-local subgroups we see in the relevant sporadic group. This requires the understanding of the module structure of the essential subgroups using local methods in group theory such as coprime action. Using this information, we understand the possibilities for a corefree fusion system. At this point, we show that we require all the local subsystems we see in the sporadic groups to generate a corefree fusion system. We complete the classification by then showing that an arbitrary corefree fusion system is isomorphic to one realized by $\Fi_{24}'$, $\Fi_{24}$ or $\Mn$. This typically makes use of \cite[Proposition 2.11]{todd-modules}, which gives us sufficient conditions to show that local subsystems are equal.

    Because the $3$-groups we are working with are rather large, we will make use of both GAP \cite{GAP4} and MAGMA \cite{magma} throughout the paper. This code can be found in the GitHub repository \cite{sporadics-code}. If a result makes use of code, we highlight a specific file associated that can be found in the repository. We further remark that much of the code illustrates properties about a Sylow $3$-subgroup of $\Fi_{24}'$ and $\Mn$ for the reader less familiar with the structure of these $3$-groups.

    We next describe the structure of the paper. Section \ref{sec:background} goes through some background on group theory and fusion systems used in this paper. In Section \ref{sec:alg}, we expand the theory of proto-essential subgroups. We characterize some relevant subgroups of $\SL_{7}(3)$, $\Sp_{10}(3)$ and $\Sp_{12}(3)$ in Section \ref{sec:automizer}. We classify all corefree fusion systems on a Sylow $3$-subgroup of $\Fi_{24}'$ and $\Mn$ in sections \ref{sec:fi24} and \ref{sec:m} respectively. Finally, Appendix \ref{sec:algorithm-appendix} looks at some applications of the algorithms described in the paper.

    The notation used in this paper is mostly standard. We write maps on the right. Groups of Lie type are referred to by their name in the ATLAS \cite{atlas}. We mainly follow the notations in \cite{ks06} for group theory, and \cite{cra} or \cite{ako} for fusion systems.

    \section{Background} \label{sec:background}
    In this section, we go through some background in group theory and fusion systems that we shall make use of in this paper. We also establish some further results that we need. 
    
    \subsection{Group Theory} We make use of standard results in group theory, as can be found in \cite{ks06}. We list the key results that we continually use in this work. Throughout, let $G$ be a finite group.
    
    \begin{lemma}[Three Subgroups Lemma]
        Let $X, Y, Z \leq G$ be such that $[X,Y,Z] = 1$ and $[Y,Z,X]=1$. Then $[Z,X,Y]=1$.
    \end{lemma}
    \begin{proof}
        This is proven in \cite[1.5.6]{ks06}.
    \end{proof}

    \begin{definition}
        Let $S$ be a finite $p$-group. We define the \emph{Thompson subgroup} of $S$, denoted $J(S)$, to be the subgroup of $S$ generated by elementary abelian subgroups of largest order.
    \end{definition}

    \begin{lemma} \label{lm:grp-thompson}
        Let $S$ be a finite $p$-group. Then the following hold:
        \begin{enumerate}
            \item $J(S)$ is a characteristic subgroup of $S$.
            \item If $J(S) \leq P \leq S$, then $J(P) = J(S)$.
        \end{enumerate}
    \end{lemma}
    \begin{proof}
        This is \cite[9.2.8]{ks06}.
    \end{proof}

    We shall make use of the following results throughout the paper without explicit reference.
    \begin{lemma}[Coprime Action]
        Let $G$ act coprimely on $A$, and let $B$ be a $G$-invariant subgroup of $A$. Then the following hold:
        \begin{enumerate}
            \item $C_{A/B}(G) = C_A(G)B/B$;
            \item if $G$ acts trivially on $A/B$ and $B$, then $G$ acts trivially on $A$;
            \item if $G$ acts trivially on $A/\Phi(A)$, then $G$ acts trivially on $A$;
            \item $A = [A,G] C_A(G)$ and if $A$ is abelian, $A = [A,G] \times C_A(G)$.
        \end{enumerate}
    \end{lemma}
    \begin{proof}
        (1) and (2) are given in \cite[8.2.2]{ks06}. (3) is proven in \cite[8.2.9]{ks06}. (4) is proven in \cite[8.2.7]{ks06}.
    \end{proof}

    We can strengthen (3) from the previous lemma by the following result of Burnside.
    
    \begin{lemma}[Burnside] \label{lm:burnside}
        Let $S$ be a finite $p$-group. Then $C_{\Aut(S)}(S/\Phi(S))$ is a normal $p$-subgroup of $\Aut(S)$.
    \end{lemma}
    \begin{proof}
        See \cite[Theorem I.5.1.4]{gor07}.
    \end{proof}

    We will also make use of the following two results liberally without reference.
    \begin{lemma}
        Let $G$ be a finite group that acts on $p$-group $A$, and assume that $B \leq A$ is also $G$-invariant. If $[A : B] = p$, then $[O^{p'}(G), A] \leq B$. In particular, $O^{p'}(G)$ centralizes $A/B$.
    \end{lemma}
    \begin{proof}
        Let $T \in \Syl_p(G)$. Consider the quotient $X := T/C_T(A/B)$. Then $X$ is a $p$-group that acts faithfully on $A/B$ of order $p$. In other words, $X$ is isomorphic to a subgroup of $\Aut(A/B)$. But $\Aut(A/B)$ is a $p'$-group, meaning that $X = 1$. Thus, $T = C_T(A/B)$. Taking closure in $G$, we infer that $O^{p'}(G) = C_{O^{p'}(G)}(A/B)$. We conclude that $O^{p'}(G)$ centralizes $A/B$.
    \end{proof}

    \begin{lemma} \label{lm:centric-faithful}
        Let $E$ be a finite $p$-group, with $A \normalIn E$. Let $G$ be a subgroup of $\Aut(E)$ such that $A$ is $G$-invariant. If $C_E(A) \leq A$, then $C_G(A) \leq O_p(G)$.
    \end{lemma}
    \begin{proof}
        See \cite[Lemma 2.6]{paper:paper-1}.
    \end{proof}

    Let $G$ be a group and $A \leq H \leq G$. We say that $A$ is \emph{weakly closed} in $H$ with respect to $G$ if for all $x \in G$ with $A^x \leq H$, we have that $A^x = A$.
    \begin{lemma} \label{lm:weak-closure-equality}
        Let $G$ be a finite group and let $X, Y, Z \leq G$ be such that:
        \begin{enumerate}
            \item $X$ is weakly closed in $S \in \Syl_p(Y) \cap \Syl_p(Z)$ with respect to $G$;
            \item there exists a $g \in G$ such that $Y^g = Z$; 
            \item $N_Y(X) = N_Z(X)$; and
            \item $N_G(N_Z(X)) \leq N_G(Z)$.
        \end{enumerate}
        Then $Y = Z$.
    \end{lemma}
    \begin{proof}
        See \cite[Lemma 2.7]{paper:paper-1}.
    \end{proof}

    \subsection{Fusion Systems} We refer the reader to \cite{ako} and \cite{cra} for background on fusion systems. Throughout, let $S$ be a finite $p$-group.

    Let $\calF$ be a fusion system on $S$, and let $\calE(\calF)$ denote the essential subgroups of $\calF$. The following result illustrates the importance of $\calE(\calF)$.
    \begin{theorem}[Alperin-Goldschmidt]
        Let $\calF$ be a saturated fusion system on $S$. Then 
        \[\calF = \langle \Aut_\calF(S), \Aut_\calF(E) \mid E \in \calE(\calF) \rangle_S.\]
    \end{theorem}
    \begin{proof}
        This is given in \cite[Theorem I.3.5]{ako} and \cite[Proposition 7.25]{cra}.
    \end{proof}
    
    \begin{lemma} \label{lm:fs-next}
        Let $\calF$ be a saturated fusion system on a finite $p$-group $S$ and let $R$ be fully $\calF$-normalized. Then for any $Q \in R^\calF$, there exists a map $\phi \colon N_S(Q) \to S$ such that $Q\phi = R$.
    \end{lemma}
    \begin{proof}
        See \cite[Lemma I.2.6 (c)]{ako} or \cite[Proposition 4.17]{cra}.
    \end{proof}

    \begin{lemma} \label{lm:nfa-sat}
        If $A$ is fully $\calF$-normalized and $\calF$ is saturated, then $N_\calF(A)$ is also saturated. 
    \end{lemma}
    \begin{proof}
        See \cite[I.5.6]{ako} or \cite[Theorem 4.27]{cra}.
    \end{proof}

    We say that $A \leq S$ is \emph{weakly $\calF$-closed} if for every $\phi \in \Hom_\calF(A, S)$, we have $A\phi = A$. We say that $A \leq S$ is \emph{strongly $\calF$-closed} if every $\phi \colon P \to S$ in $\calF$ with $P \leq A$ satisfies $P\phi \leq A$.
    \begin{lemma} \label{lm:weak-closure-to-normality}
        Let $\calF$ be a saturated fusion system on $S$, and let $A \leq S$. Then $A \normalIn \calF$ if and only if $A$ is weakly closed in $\calF$, and for every $E \in \calE(\calF)$, we have that $A \leq E$.
    \end{lemma}
    \begin{proof}
        This is a restatement of \cite[Lemma I.4.5]{ako}.
    \end{proof}

    \begin{lemma} \label{lm:autfe-decomposition}
        Assume that $\calF$ is a saturated fusion system on $S$, $A \leq S$ is fully $\calF$-normalized. If $\alpha \in N_{\Aut_\calF(A)}(\Aut_S(A))$, then $\alpha$ has an extension $\hat{\alpha} \in \Aut_\calF(N_S(A))$.
    \end{lemma}
    \begin{proof}
        As $A$ is fully $\calF$-normalized, we know that $A$ is $\calF$-receptive by \cite[Theorem 4.21]{cra}. In particular, the map $\alpha$ lifts to 
        \[N_\alpha := \{x \in N_S(A) \mid c_x^\alpha \in \Aut_S(A)\}.\]
        We have chosen $\alpha \in \Aut_\calF(A)$ to normalize $\Aut_S(A)$. This implies that $N_\alpha = N_S(A)$. Thus, $\alpha$ has an extension $\hat{\alpha} \colon N_S(A) \to N_S(A)$ in $\calF$.
    \end{proof}

    \begin{theorem}[Model Theorem]
        Let $\calF$ be a saturated fusion system on a finite $p$-group $S$. Let $Q \normalIn S$ be such that $Q$ is both normal and centric in $\calF$. Then the following hold:
        \begin{enumerate}
            \item There exists a model for $\calF$. That is, there exists a group $G$ with $S \in \Syl_p(G)$ such that $\calF = \calF_S(G)$ and $F^*(G) = O_p(G)$.
            \item If $G_1$ and $G_2$ are models for $\calF$, then there exists a group isomorphism $\phi \colon G_1 \to G_2$ such that $\phi|_S = 1$.
            \item If $G$ is a finite group such that $Q = F^*(G)$ and $\Aut_\calF(Q) = \Aut_G(Q)$, then there exists a $\beta \in \Aut(S)$ such that $\beta|_Q = 1$ and $\calF^\beta = \calF_S(G)$. That is, there is a model for $\calF$ that is isomorphic to $G$.
        \end{enumerate}
    \end{theorem}
    \begin{proof}
        See \cite[Theorem I.4.9]{ako}.
    \end{proof}

    We shall use the following result to compare essential subgroups in the entire fusion system with those in normalizer subsystems (cf. \cite[Chapter 4.4]{cra}).
    
    \begin{lemma} \label{lm:essentials-in-normalizer}
        Let $\calF$ be a saturated fusion system on $S$, and let $A \leq S$ be weakly $\calF$-closed. Then 
        \[\calE(N_\calF(A)) = \{E \in \calE(\calF) \mid A \leq E\}.\]
    \end{lemma}
    \begin{proof}
        See \cite[Lemma 3.11]{paper:paper-1}.
    \end{proof}
        
    \begin{definition}
        Let $\calF$ be a saturated fusion system on $S$, and let $A \normalIn \calF$. Define the quotient map $\pi \colon S \to S/A$. We define the \emph{quotient fusion system} $\calF/A$ on $S/A$, where for subgroups $X$ and $Y$ of $S$ containing $A$, we set
        \[\Hom_{\calF/A}(X/A, Y/A) := \{\psi^\pi \mid \psi \in \Hom_\calF(X,Y)\}.\]
    \end{definition}

    \begin{lemma}
        Let $\calF$ be a saturated fusion system on $S$, and let $A \normalIn \calF$. Then $\calF/A$ is saturated.
    \end{lemma}
    \begin{proof}
        See \cite[Proposition 5.11]{cra}.
    \end{proof}
    
    \section{Finding all proto-essential subgroups of a $p$-group} \label{sec:alg}
    In this section, we present a novel algorithm to find all proto-essential subgroups of a finite $p$-group $S$. A subgroup $E$ of $S$ is said to be \emph{proto-essential} if there could exist a saturated fusion system $\calF$ on $S$ such that $E \in \calE(\calF)$ (see \cite[Appendix A]{paper:paper-1} for a definition). As we make use of proto-essential subgroups to classify all fusion systems $\calF$ on $S$, the definition is independent of $\calF$.
    
    In the MAGMA implementation of fusion systems described in \cite{algorithms}, the authors start by computing all subgroups of $S/Z(S)$ and then filter those that are proto-essential. Most of the subgroups that are computed are not proto-essential. Indeed, the process of finding all subgroups is a limiting factor in the classification of fusion systems on a given $p$-group. Moreover, testing that a subgroup is proto-essential is rather expensive as it requires the construction of its automorphism group (which is typically not solvable if $p \geq 5$).

    The algorithm we describe roughly works as follows. We fix a normal series of $S'$
    \[1 = Z_0 \normalIn Z_1 \normalIn \dots \normalIn Z_n = S',\]
    where $[Z_{i+1} : Z_i] = p$ for each $0 \leq i < n$, where each $Z_i$ is normal in $S$. We identify subgroups in $S/Z_i$ of the form $A := C_{S/Z_i}(x)$ for some $x \in S/Z_i$ of order $p$, and consider the full preimage $\hat{A}$ of $A$ in $S$. We show that in any fusion system $\calF$ on $S$ and $E \in \calE(\calF)$, there is some choice of $\hat{A}$ that is proto-essential in $S$ such that $E$ and $\hat{A}$ are $\calF$-conjugate. Using this information, we can find all proto-essential subgroups of $S$. 
    
    This method considers much fewer subgroups than the Parker-Semeraro algorithm. As a result, the algorithm runs much faster in practice. We further expand this algorithm to show that if there is no proto-essential subgroup of the form $C_S(x)$ for some $x \in S$ of order $p$, then $S$ cannot support corefree fusion systems.

    \begin{lemma} \label{lm:cfz-nfs-nfz}
        Let $\calF$ be a saturated fusion system on a finite $p$-group $S$, and let $Z \leq Z(S)$. Then
        \[C_\calF(Z) \subseteq N_\calF(Z) \subseteq \langle N_\calF(S), C_\calF(Z) \rangle.\]
    \end{lemma}
    \begin{proof}
        By definition, we see that $C_\calF(Z) \subseteq N_\calF(Z)$. Let $\calF_0 := \langle N_\calF(S), C_\calF(Z) \rangle$ (which need not be saturated) and $\gamma \colon A \to B$ in $N_\calF(Z)$. By definition, there exists a map $\alpha \colon AZ \to BZ$ extending $\gamma$ such that $\alpha|_Z \in \Aut_\calF(Z)$. Since $C_S(Z) = S$, there exists a map $\hat{\alpha} \in \Aut_\calF(S)$ that extends $\alpha|_Z$. We note that the map $\hat{\alpha}$ lies in $\calF_0$. Moreover, the composition $\alpha^{-1} \circ \hat{\alpha} \colon BZ \to S$ lies in $C_\calF(Z) \subseteq \calF_0$. We infer that $\alpha$ itself lies in $\calF_0$. The result now follows since $\gamma = \alpha|_A$.
    \end{proof}

    \begin{lemma} \label{lm:alg_two_gen} 
        Let $\calF$ be a saturated fusion system on a finite $p$-group $S$, and let $Z \leq Z(S)$ have order $p$. Let $\calG$ be a fusion system (not necessarily saturated) on $S$ satisfying 
        \[\langle N_\calF(S), C_\calF(Z) \rangle \subseteq \calG \subseteq \calF.\]
        If $\calF \neq \calG$, then there exists an $E \in \calE(\calF)$ such that $\Aut_\calF(E) \neq \Aut_\calG(E)$. Moreover, if $E \in \calE(\calF)$ has maximal order amongst all subgroups $E_0 \in \calE(\calF)$ satisfying $\Aut_\calG(E_0) \neq \Aut_\calF(E_0)$, then there exists an $x \in S$ of order $p$ such that $E = C_S(x)$.
    \end{lemma}
    \begin{proof}
        We adapt the proof in \cite[Lemma 2.3]{todd-modules}. Since $\Aut_\calF(S) = \Aut_\calG(S)$, we see that if $\calF \neq \calG$, then there exists an $E \in \calE(\calF)$ such that $\Aut_\calF(E) \neq \Aut_\calG(E)$ by Alperin-Goldschmidt. Assume now that $E$ has maximal order amongst all subgroups $E_0 \in \calE(\calF)$ satisfying $\Aut_\calG(E_0) \neq \Aut_\calF(E_0)$.
        
        Let $\alpha \in \Aut_\calF(E)$ be a map that does not lie in $\calG$. Set $X := Z \alpha$. We know that $Z$ is contained in $Z(S) \leq Z(E)$. As $\alpha$ normalizes $Z(E)$, we infer that $X$ is a central subgroup of $E$. In particular, 
        \[E = C_E(X) \leq C_S(X) \leq N_S(X).\]
        Set $x$ to be a generator of $X$, a group of order $p$.

        Since $Z \leq Z(S)$, we find that $Z$ is fully normalized in $\calF$. By Lemma \ref{lm:fs-next}, we can find a map $\psi \colon N_S(X) \to S$ in $\calF$ such that $X\psi = Z$. Consider now the map $\gamma \colon E \to S$ given by $\gamma := \alpha \circ \psi|_E$. Since $\gamma$ normalizes $Z \leq Z(S)$, we know by Lemma \ref{lm:cfz-nfs-nfz} that $\gamma$ lies in $\calG$. On the other hand, $\alpha$ does not lie inside $\calG$. We deduce that $\psi|_E$, and so $\psi$, is also not an element of $\calG$.

        Applying Alperin-Goldschmidt to $\calF$, we can write
        \[\psi = \theta_1 \circ \theta_2 \dots \circ \theta_n,\]
        with $\theta_i \in \Aut_\calF(E_i)$ for $E_i \in \calE(\calF) \cup \{S\}$. Since $\psi$ is not an element of $\calG$, we infer that some $\theta_i$ does not lie in $\calG$. Since $N_\calF(S) \subseteq \calG$, we have that $E_i \in \calE(\calF)$. That is, $E_i$ lies in $\calE(\calF)$ and satisfies $\theta_i \in \Aut_\calF(E_i) \setminus \Aut_\calG(E_i)$. By the maximality of $E$, we conclude that
        \[|E| \leq |N_S(X)| \leq |E_i| \leq |E|.\]
        Thus, $E = N_S(X) = C_S(X) = C_S(x)$.
    \end{proof}    
    
    \begin{corollary} \label{cor:csx}
        Let $\calF$ be a saturated fusion system on a finite $p$-group $S$, and let $Z \leq Z(S)$ be of order $p$. For $E \in \calE(\calF)$, if $E^\calF \cap \calE(C_\calF(Z)) = \varnothing$, then there exists an $x \in S$ of order $p$ such that $E = C_S(x)$.
    \end{corollary}
    \begin{proof}
        Let $\calE_0 := \calE(C_\calF(Z)) \cup (\calE(\calF) \setminus E^\calF)$, and define
        \[\calG := \langle \Aut_\calF(S), \Aut_\calF(E_0) \mid E_0 \in \calE_0 \rangle.\]
        Assume, for a contradiction, that $\Aut_\calF(E) = \Aut_\calG(E)$. Then every $\alpha \in \Aut_\calF(E)$ is of the form
        \[\alpha = \theta_1 \circ \dots \circ \theta_n,\]
        where $\theta_i \in \Aut_\calF(E_i)$ for $E_i \in \calE_0 \cup \{S\}$. By construction, we have that $E^\calF \cap \calE_0 = \varnothing$. This implies that $|E_i| > |E|$ for each $1 \leq i \leq n$. By \cite[Proposition I.3.3 (b)]{ako}, we conclude that $E \not\in \calE(\calF)$, a contradiction. 
        
        Since $\Aut_\calF(E) \neq \Aut_\calG(E)$, we know that $\calF \neq \calG$. As $\calE(C_\calF(Z)) \subseteq \calE_0$, we see that $\langle C_\calF(Z), N_\calF(S) \rangle \subseteq \calG$. Now if $E_0 \in \calE(\calF)$ is not $\calF$-conjugate to $E$, then we have $E_0 \in \calE_0$ by construction. This implies that $\Aut_\calF(E_0) = \Aut_\calG(E_0)$. In other words, if $E_0 \in \calE(\calF)$ is such that $\Aut_\calF(E_0) \neq \Aut_\calG(E_0)$, then $E_0 \in E^\calF$. Since every $\calF$-conjugate of $E$ has the same order, we may apply Lemma \ref{lm:alg_two_gen} to deduce that $E = C_S(x)$ for some $x \in S$ of order $p$.
    \end{proof}

    Let $\calF$ be a fusion system on a finite $p$-group, and let $Z \normalIn S$. We define the subsystem $SC_\calF(Z)$ of $\calF$ on $S$, where a map $\alpha \in \Hom_\calF(X, Y)$ lies in $\Hom_{SC_\calF(Z)}(X, Y)$ if and only if there exists an extension $\hat{\alpha} \colon AZ \to BZ$ of $\alpha$ in $\calF$ such that $\hat{\alpha}|_Z \in \Aut_S(Z)$. We know by \cite[Corollary 4.38]{cra} that $SC_\calF(Z)$ is saturated whenever $\calF$ is saturated. Moreover, if $Y \leq Z \leq S$ are subgroups such that both $Y$ and $Z$ are normal in $S$, then $SC_\calF(Y) \subseteq SC_\calF(Z)$.

    \begin{lemma} \label{lm:essential-central}
        Let $\calF$ be a saturated fusion system on $S$, and let $Z \normalIn S$ be such that $\calF = SC_\calF(Z)$. Set $\overline{S} := S/Z$ and $\overline{\calF} := SC_\calF(Z)/Z$. If $Z \leq E \leq S$ is a subgroup, then $E \in \calE(SC_\calF(Z))$ if and only if $\overline{E} \in \calE(\overline{\calF})$.
    \end{lemma}
    \begin{proof}
        See \cite[Section 3]{kl-centralext} or \cite[Proposition 5.68]{cra}.
    \end{proof}

    \begin{lemma} \label{lm:trunc-chain}
        Let $S$ be a finite $p$-group, and let $\calF$ be a saturated fusion system on $S$. Let
        \[1 = Z_0 \leq Z_1 \leq \dots \leq Z_n = S\]
        be a central series of $S$. Define $\calF_i := SC_\calF(Z_i)$ for $0 \leq i \leq n$. Then
        \[\calF_{i+1}/Z_{i+1} = C_{\calF_i/Z_i}(Z_{i+1}/Z_i)/(Z_{i+1}/Z_i).\]
    \end{lemma}
    \begin{proof}
        As $Z_0 = 1$, we see that
        \[C_{\calF_0/Z_0}(Z_1/Z_0)/(Z_1/Z_0) = C_\calF(Z_1)/Z_1 = \calF_1/Z_1.\]
        
        Let $\phi \colon A \to B$ be a map $\calF_{i+1}$ with $Z_{i+1} \leq A$. By definition, we have that $\phi|_{Z_{i+1}} \in \Aut_S(Z_{i+1})$. Write $\overline{S} := S/Z_i$ and $\overline{\calF_i} := \calF_i/Z_i$. Then $\overline{\phi}$ is such that $\overline{\phi}|_{\overline{Z_{i+1}}} \in \Aut_{\overline{S}}(\overline{Z_{i+1}})$. By assumption, we have that $\overline{Z_{i+1}} \leq Z(\overline{S})$. Thus, $\overline{\phi}$ centralizes $\overline{Z_{i+1}}$. We deduce that $\overline{\phi}$ lies in $C_{\overline{\calF_i}}(\overline{Z_{i+1}})$. In particular,
        \[\calF_{i+1}/Z_{i+1} \subseteq C_{\overline{\calF_i}}(\overline{Z_{i+1}})/\overline{Z_{i+1}}.\]

        Now let $\phi \colon A \to B$ be a map in $\calF_i$ with $Z_{i+1} \leq A$ such that $\overline{\phi} \colon \overline{A} \to \overline{B}$ lies in $C_{\overline{\calF_i}}(\overline{Z_{i+1}})$. Since $\calF_i \subseteq \calF_j$ for $j < i$, we deduce by induction that $\phi$ centralizes $Z_{j+1}/Z_j$ for $0 \leq j < i$. Furthermore, $\phi$ centralizes $Z_{i+1}/Z_i$ by assumption. Let $X$ be the subgroup of $\Aut_\calF(Z_{i+1})$ consisting of maps that centralize $Z_{j+1}/Z_j$ for each $0 \leq j \leq i$. By coprime action, we know that $X$ is a $p$-group. Since $S$ centralizes $Z_{j+1}/Z_j$ for each $0 \leq j \leq i$, we infer that $\Aut_S(Z_{i+1}) \leq X$. Moreover, $Z_{i+1}$ is fully normalized in $\calF$, meaning that $\Aut_S(Z_{i+1})$ is a Sylow $p$-subgroup of $\Aut_\calF(Z_{i+1})$. This forces $X = \Aut_S(Z_{i+1})$. As $\phi|_{Z_{i+1}} \in X$, we conclude that $\phi$ lies in $SC_\calF(Z_{i+1})$. Thus,
        \[C_{\overline{\calF_i}}(\overline{Z_{i+1}})/\overline{Z_{i+1}} \subseteq \calF_{i+1}/Z_{i+1}\]
        as well, and the result holds.
    \end{proof}
    
    Let $S$ be a finite $p$-group. We denote by $\mathfrak{P}(S)$ the set of proto-essential subgroups of $S$. Note that for any fusion system $\calF$ on $S$, we have that $\calE(\calF) \subseteq \mathfrak{P}(S)$.
    \begin{theorem} \label{thm:algo}
        Let $\calF$ be a saturated fusion system on a finite $p$-group $S$. Take a central series of $S'$:
        \[1 = Z_0 \leq Z_1 \leq \dots \leq Z_n = S'\]
        such that $[Z_{i+1} : Z_i] = p$ for $0 \leq i < n$, where each $Z_i$ is normal in $S$. Define the projection maps $\pi_i \colon S \to S/Z_i$ for $0 \leq i < n$, and
        \[\mathfrak{C}_i := \{\pi_i^{-1}(E) \mid E := C_{S/Z_i}(x), x \in S/Z_i \textrm{ order } p \} \cap \mathfrak{P}(S).\]
        For $E \in \calE(\calF)$, there exists a $\calF$-conjugate $E_0$ of $E$ such that $E_0 \in \mathfrak{C}_i$ for some $0 \leq i < n$.
    \end{theorem}
    \begin{proof}
        For $0 \leq i \leq n$, set $\calF_i := SC_\calF(Z_i)$. Then we have the following series of subsystems of $\calF$:
        \[\calF_n \subseteq \calF_{n-1} \subseteq \dots \subseteq \calF_1 \subseteq \calF_0 = \calF.\]
        Set $\calE_i := \calE(\calF_i)$. As $\calF_n/Z_n$ is a fusion system on an abelian group $S/Z_n$, we know by Lemma \ref{lm:essential-central} that $\calE_n = \varnothing$. Let $E \in \calE_0 = \calE(\calF)$. Then there exists an $0 \leq i < n$ such that $\calE_i \cap E^\calF \neq \varnothing$ but $\calE_{i+1} \cap E^\calF = \varnothing$.

        Let $\overline{\calF_i} := \calF_i/Z_i$ and $\overline{S} := S/Z_i$. We know by Lemma \ref{lm:trunc-chain} that 
        \[C_{\overline{\calF_i}}(\overline{Z_{i+1}})/\overline{Z_{i+1}} = \calF_{i+1}/Z_{i+1}.\]
        By the definition of $i$, there exists some $E_0 \in \calE_i \cap E^\calF$. Since $E^\calF \cap \calE_{i+1} = \varnothing$, we can apply Lemma \ref{lm:essential-central} twice to deduce that 
        \[\overline{E_0}^{\overline{\calF_i}} \cap \calE(C_{\overline{\calF_i}}(\overline{Z_{i+1}})) = \varnothing.\]
        Applying Corollary \ref{cor:csx}, we see that there exists an $x \in \overline{S}$ such that $\overline{E_0} = C_{\overline{S}}(x)$. Since $E_0 \in \calE_i = \calE(SC_\calF(Z_i))$, we know that $Z_i \leq E_0$. Furthermore, $SC_\calF(Z_i)$ is a saturated fusion system on $S$, meaning that $E_0 \in \mathfrak{P}(S)$. We conclude that $E_0 \in \mathfrak{C}_i$, as desired.
    \end{proof}
    
    We let $\mathfrak{C}$ be the union of $\mathfrak{C}_i$ for $0 \leq i < n$. The set $\mathfrak{C}$ can be computed as follows.
    \begin{algorithm} \label{algorithm}
        Let $S$ be a finite $p$-group, and take a central series of $S'$:
        \[1 = Z_0 \leq Z_1 \leq \dots \leq Z_n = S'\]
        such that $[Z_{i+1} : Z_i] = p$ for $0 \leq i < n$, where each $Z_i$ is normal in $S$. Define the projection maps $\pi_i \colon S \to S/Z_i$. We construct the set $\mathfrak{C}$ as follows:
        \begin{enumerate}
            \item For $0 \leq i < n$, let $L_i$ denote the set of elements of $S/Z_i$.
            \item For each $x \in L_i$ of order $p$, compute $E_0 := C_{S/Z_i}(x)$, and add $E := \pi_i^{-1}(E_0)$ to $\mathfrak{C}'$.
            \item For a subgroup $E \in \mathfrak{C}'$, if $E$ is proto-essential in $S$, we add $E$ to $\mathfrak{C}$.
        \end{enumerate}
    \end{algorithm}

    An immediate consequence of Theorem \ref{thm:algo} is the following corollary.
    \begin{corollary} \label{cor:alg}
        Let $\calF$ be a saturated fusion system on a finite $p$-group $S$. If $E \in \calE(\calF)$, then a $\calF$-conjugate of $E$ lies in the set $\mathfrak{C}$ computed by Algorithm \ref{algorithm}.
    \end{corollary}
    
    We note that any subgroup $C := C_{S/Z_i}(x)$ is self-centralizing in $S/Z_i$. This means that the full preimage $C_0$ must also be self-centralizing in $S$. As such, we do not need to perform this test when checking whether a subgroup in $\mathfrak{C}$ is proto-essential (see \cite[Appendix A]{paper:paper-1} for all the tests we perform when checking whether a subgroup is proto-essential). Appendix \ref{sec:algorithm-appendix} illustrates how the new algorithm runs on a Sylow $3$-subgroup of $\Fi_{22}$, $\Fi_{23}$, $\Fi_{24}'$ and $\Mn$, and also compares it to the previous algorithm in the first two cases.

    Let $\calF$ be a saturated fusion system on a finite $p$-group $S$, and let $\calF_i := SC_\calF(Z_i)$. For $0 \leq j \leq i$, we have that $\calF_i = SC_{\calF_i}(Z_j)$. This implies that, if $E \in \calE(\calF_i)$, then $E/Z_j \in \calE(\calF_i/Z_j)$ by Lemma \ref{lm:essential-central}. As such, in addition to step (3), we can check whether $E/Z_j$ is proto-essential in $S/Z_j$. We perform this test only if we believe that step (3) is too expensive. We remark that this can only allow us to say that $E$ is not proto-essential in $S$. In other words, if $E/Z_j$ is proto-essential in $S/Z_j$, this does not imply that $E$ is proto-essential in $S$. For example, if $S \cong 3 \wr 3$ and $j = 1$, then we have $Z_j = Z(S)$. In particular, all four maximal subgroups of $S/Z_j$ are proto-essential in $S/Z_j \cong 3^{1+2}_+$. However, the two maximal subgroups of $S$ of exponent $9$ are not proto-essential in $S$ (see, for example, \cite[Lemma 3.2]{thesis:mm}).

    We further remark that an argument similar to Theorem \ref{thm:algo} does not hold for radical subgroups. For example, let $\calF$ be the $3$-fusion category of $\O_7(3)$ (see \cite[Section 5]{paper:paper-1}). Then there exists a subgroup $A := 3^{3+3}$ that is $\calF$-centric radical. But for any normal subgroup $N$ of $S$ contained in $A$, we find that $A/N \neq C_{S/N}(x)$ for any element $x \in S/N$ (not necessarily of order $p$). 

    We can compute the set of proto-essential subgroups of $S$ using $\mathfrak{C}$. The set $\mathfrak{C}$ finds the subgroups up to $\calF$-conjugacy, but we want to find all the proto-essential subgroups $\mathfrak{P}(S)$, which is independent of $\calF$.
    \begin{algorithm} \label{alg:all-proto}
        Let $S$ be a finite $p$-group, and let $\mathfrak{C}$ be a subset of proto-essential subgroups of $S$, as computed by Algorithm \ref{algorithm}. We construct the set $\mathfrak{C}^*$ as follows:
        
        While $\mathfrak{C}$ is non-empty:
        \begin{enumerate}
            \item remove $E_0 \in \mathfrak{C}$ of largest order, and add every subgroup in $E_0^{\Aut(S)}$ to $\mathfrak{C}^*$ if $E_0$ is proto-essential;
            \item for each $X_i \in \mathfrak{C}^*$ containing $E_0$, add the set $E_0^{\Aut(X_i)}$ to $\mathfrak{C}$.
        \end{enumerate}
    \end{algorithm}
    An implementation of Algorithms \ref{algorithm} and \ref{alg:all-proto} in GAP can be found in the file \texttt{find-proto-essentials.g}. In both cases, we adapt the algorithms given above to compute the relevant subgroups up to $\Aut(S)$-conjugacy as being proto-essential is an $\Aut(S)$-invariant property. This ensures that the proto-essential test is run for as few subgroups as possible. Moreover, this is typically not an expensive process if $\Aut(S)$ is solvable.
    
    It is important that in step (1), the subgroup $E_0$ from $\mathfrak{C}$ is chosen to have largest order. This ensures that, if at any point when the algorithm is running on $E_0$, any subgroup of larger order that lies in $\mathfrak{C}^*$ at the end of the algorithm has already been added to $\mathfrak{C}^*$. This, in particular, includes subgroups that properly contain $E_0$.

    \begin{corollary} \label{cor:main-alg}
        Let $S$ be a finite $p$-group. Then the set $\mathfrak{C}^*$ computed in Algorithm \ref{alg:all-proto} consists precisely of the proto-essential subgroups of $S$.
    \end{corollary}
    \begin{proof}
        Note that we only add subgroups that are proto-essential in $S$ to the set $\mathfrak{C}^*$. This implies that $\mathfrak{C}^* \subseteq \mathfrak{P}(S)$. Assume now that $\mathfrak{C}^* \subsetneq \mathfrak{P}(S)$. Then there exists some finite $p$-group $S$, a saturated fusion system $\calF$ on $S$, and an essential subgroup $E \in \calE(\calF)$ such that $E \not\in \mathfrak{C}^*$. We may further assume that every essential subgroup of $\calF$ that properly contains $E$ lies in $\mathfrak{C}^*$.

        By Theorem \ref{thm:algo}, we know that there exists an isomorphism $\phi \colon E_0 \to E$ in $\calF$, for some $E_0 \in \mathfrak{C}$. We appeal to Alperin-Goldschmidt to decompose $\phi$:
        \[\phi = \phi_1 \circ \dots \circ \phi_n,\]
        where each $\phi_i \in \Aut_\calF(F_i)$ for some $F_i \in \calE(\calF) \cup \{S\}$. If each $F_i$ is equal to $S$, then $E$ and $E_0$ are $\Aut(S)$-conjugate. But then $E \in E_0^{\Aut(S)}$, where $E_0 \in \mathfrak{C}$. This means that $E$ gets added to $\mathfrak{C}^*$, which is a contradiction.

        We write $E_i := E_0 \phi_1 \circ \dots \circ \phi_i$ for $1 \leq i \leq n$, so that $E = E_n$. Fix $j \geq 1$ to be minimal such that $F_j \neq S$. Then $E_0$ and $E_{j-1}$ are $\Aut(S)$-conjugate. If $\alpha = \phi_1 \circ \dots \circ \phi_{j-1}$, then we have the map $\psi := \phi^{\alpha^{-1}}$ in $\calF$ that maps $E_0$ to $E_j^{\alpha^{-1}}$. This is well-defined as $\alpha \in \Aut_\calF(S)$. As $\psi \in \Aut(F_j)$ and $F_j$ lies in $\mathfrak{C}^*$, we see that $E_j^{\alpha^{-1}}$ gets added to $\mathfrak{C}$ in step (3). If $E_j^{\alpha^{-1}}$ is $\Aut(S)$-conjugate to $E$, then we get a contradiction. Otherwise, there exists some $k > j$ minimal such that $F_k \neq S$ as well. Then we find that, during the iteration on $E_j^{\alpha^{-1}}$, some $\Aut_\calF(S)$-conjugate of $E_k$ gets added to $\mathfrak{C}$. Continuing the argument, we will see that eventually, some $\Aut_\calF(S)$-conjugate $E_*$ of $E_n = E$ gets added to $\mathfrak{C}$. As $E$ is proto-essential, $E_*$ is also proto-essential. This implies that $E \in E_*^{\Aut(S)}$ gets added to $\mathfrak{C}^*$, giving us the final contradiction. We conclude that $\mathfrak{C}^* = \mathfrak{P}(S)$.
    \end{proof}

    We remark that step (2) in Algorithm \ref{alg:all-proto} can be quite expensive. Nonetheless, we can avoid this step in many cases and still find $\mathfrak{P}(S)$. We consider two cases when this holds:
    \begin{itemize}
        \item if $E_0$ is maximal amongst the proto-essential subgroups in $S$ with respect to inclusion; or
        \item if $E_0$ is the unique subgroup of $S$ satisfying some property, up to $\Aut(S)$-conjugacy.
    \end{itemize}
    In these cases, there is no choice for the subgroup $X_i$ in step (2). We consider one further optimisation when $E_0$ is normal in $S$. In that case, as $E_0$ is $\calF$-conjugate to some essential subgroup $E$. As $E$ is fully $\calF$-normalized, we deduce that $E$ must also be normal in $S$. Thus, it suffices to consider normal subgroups of $S$ that can be $\Aut(X_i)$-conjugate to $E_0$ in step (2).

    We end this section by considering how we can use $\mathfrak{C}_0$ to detect whether a $p$-group can support corefree fusion systems.
    \begin{lemma} \label{lm:f-ozs}
        Let $\calF$ be a saturated fusion system on $S$, and let $\calG := N_\calF(\Omega(Z(S)))$. If $\calF \neq \calG$, then there exists an $E \in \calE(\calF)$ such that $\Aut_\calF(E) \neq \Aut_\calG(E)$ and $E = C_S(x)$ for some $x \in S$ of order $p$.
    \end{lemma}
    \begin{proof}
        Since $\Aut_\calF(S)$ normalizes $\Omega(Z(S))$, we know that $\Aut_\calF(S) = \Aut_\calG(S)$. As such, if $\calF \neq \calG$, then there exists an $E \in \calE(\calF)$ such that $\Aut_\calF(E) \neq \Aut_\calG(E)$ by Alperin-Goldschmidt. Assume now that $E$ has maximal order amongst all subgroups $E_0 \in \calE(\calF)$ satisfying $\Aut_\calG(E_0) \neq \Aut_\calF(E_0)$.
        
        Let $\alpha \in \Aut_\calF(E)$ be a map that does not lie in $\calG$, and set $A := \Omega(Z(S)) \alpha$. By construction, we know that $A \neq \Omega(Z(S))$. But $\alpha$ does normalize $\Omega(Z(E))$, meaning that $A$ is a central subgroup of $E$ of exponent $p$. We now take some $x \in A \setminus \Omega(Z(S))$, which has order $p$. Set $X := \langle x \rangle$, $Z := \langle x\alpha^{-1} \rangle \leq \Omega(Z(S))$ and $\beta := \alpha^{-1}|_X$, so that $X\beta = Z$. As $x \in Z(E)$, we infer that
        \[E = C_E(x) \leq C_S(x) = C_S(X).\]
        
        The map $\beta$ lies in $\calF$ and maps $X$ to $Z \leq Z(S)$. Thus, there exists a map $\hat{\beta} \colon C_S(X) \to S$ in $\calF$ such that $\hat{\beta}|_X = \beta$. We further see that $X \nleq \Omega(Z(S))$ satisfies $X \hat{\beta} = X\beta = Z \leq \Omega(Z(S))$. In other words, $\hat{\beta}$ cannot lie in $\calG = N_\calF(\Omega(Z(S)))$.

        By Alperin-Goldschmidt, we can write
        \[\hat{\beta} = \theta_1 \circ \theta_2 \dots \circ \theta_n,\]
        with $\theta_i \in \Aut_\calF(E_i)$ for $E_i \in \calE(\calF) \cup \{S\}$. Since $\hat{\beta}$ is not an element of $\calG$, we infer that some $\theta_i$ does not lie in $\calG$. As $\Aut_\calF(S) = \Aut_\calG(S)$, we have that $E_i \in \calE(\calF)$. That is, $E_i$ lies in $\calE(\calF)$ and satisfies $\theta_i \in \Aut_\calF(E_i) \setminus \Aut_\calG(E_i)$. By the maximality of $E$, we conclude that
        \[|E| \leq |C_S(X)| \leq |E_i| \leq |E|.\]
        Thus, $E = C_S(X) = C_S(x)$.
    \end{proof}

    \noindent The following is an immediate consequence of Lemma \ref{lm:f-ozs}.
    \begin{corollary} \label{cor:opf}
        Let $S$ be a finite $p$-group. If $\mathfrak{C}_0 = \varnothing$, then for any saturated fusion system $\calF$ on $S$, we have that $\Omega(Z(S)) \normalIn \calF$.
    \end{corollary}
    Corollary \ref{cor:opf} gives us quite an efficient way of checking whether a $p$-group $S$ can support corefree fusion systems -- we check whether $\mathfrak{C}_0 \neq \varnothing$. To do so, we check whether subgroups of the form $C_S(x)$ can be proto-essential in $S$, where $x \in S$ has order $p$. In practice, if $|S| \leq p^6$, then this means that we only need to check whether less than $5$ subgroups of $S$ are proto-essential to deduce whether it can support corefree fusion systems (cf. Appendix \ref{sec:algorithm-appendix} for a more detailed analysis). 
    
    Note however that $\mathfrak{C}_0 \neq \varnothing$ does not imply that $S$ supports corefree fusion systems. For example, if $S := \mathbf{SmallGroup}(5^4, 8)$, then $S$ is a Sylow $5$-subgroup of $G := \SL_2(\Z/25\Z)$. The $5$-fusion category of $G$ on $S$ has $R := C_S(Z_2(S))$ (which is elementary abelian of order $5^3$) as the unique essential subgroup. This subgroup lies in $\mathfrak{C}_0$. But we compute that $\mathfrak{P}(S) = \{R\}$, meaning that for any fusion system $\calF$ on $S$, we have that $R \leq O_5(\calF)$.

    We finally note some subgroups we typically see in $\mathfrak{C}_0$. Let $R := C_S(Z_2(S))$. For small groups, we typically see that if $\calF$ is a corefree fusion system on $S$, then $\Aut_\calF(R)$ does not centralize $Z(S)$ (see, for example, the fusion systems given in \cite[Appendix A]{algorithms}). In many cases, we further have that $R = C_S(x)$ for some $x \in Z_2(S) \setminus Z(S)$ of order $p$. Another class of subgroups lying in $\mathfrak{C}_0$ are abelian essential subgroups $V$, in which case there typically exists some $x \in V$ of order $p$ such that $V = C_S(x)$. This includes abelian pearls (as defined by Grazian in \cite{pearls}).

    \section{Some subgroups of $\SL_{n}(3)$ and $\Sp_{2n}(3)$} \label{sec:automizer}
    In this section, we will analyse certain subgroups of $\SL_n(3)$ and $\Sp_{2n}(3)$ based on their Sylow $3$-subgroup and some further assumptions. The determination of the subgroups satisfying these properties is important to understand the structure of the automizers of radical subgroups in both $\Fi_{24}'$ and $\Mn$.
    
    \begin{lemma} \label{lm:aut-sp10}
        Let $Q := 3^{1+10}_+$ be the extraspecial group of exponent $3$ and order $3^{11}$, $G \leq \Out(Q)$ and $V := Q/Z(Q)$. Set $S \in \Syl_3(G)$. Assume that:
        \begin{enumerate}
            \item $O^{3'}(G) = G$ and $O_3(G) = 1$;
            \item $S \cong 3 \times 3 \wr 3$; 
            \item $|C_V(S)| = 3$; and 
            \item there exists a normal subgroup $N \cong 3^{1+2}_+$ of $S$ such that $|C_V(A)| = 3^4$ for some complement $A$ of $N$ in $S$.
        \end{enumerate}
        Then $G \cong \U_5(2)$, and $G$ acts irreducibly on $V$.
    \end{lemma}
    \begin{proof}
        All code used in this proof can be found in \texttt{fi24/sp10.m} and \texttt{fi24/sp10.g}. Since $O^{3'}(G) = G$, we know by \cite{extraspecial-autgrp-sp} that $G$ is isomorphic to a subgroup of $\Sp_{10}(3)$. The maximal subgroups of $\Sp_{10}(3)$ are given in \cite[Tables 8.64 and 8.65]{maximals}. Using MAGMA, we compute those maximal subgroups $M$ such that $|C_V(T)| = 3$ and $T$ has a subgroup isomorphic to $3 \times 3 \wr 3$ for some $T \in \Syl_3(M)$. These are given below:
        \begin{table}[H]
            \centering
            \begin{tabular}{ccc}
                $M_1 := 2 \times 3^{1+8}_+ : \Sp_8(3)$ & $M_2 := 3^{3+12} : (\GL_2(3) \times \Sp_6(3))$ & $M_3 := 3^{15} : \GL_5(3)$ \\
                $M_4 := 3^{10+8} : (\GL_4(3) \times \Sp_2(3))$ & $M_5 := 3^{6+12} : (\GL_3(3) \times \Sp_4(3))$ & $M_6 := 2 \times \U_5(2)$ \\
                & $M_7 := \Sp_2(3) \circ \GO_5(3)$.
            \end{tabular}
        \end{table}

        Since $O^{3'}(G) = G$ and $O_3(G) = 1$, we find that if $G$ is contained in some subgroup $M_j$ for $1 \leq j \leq 5$, then $G$ is isomorphic to a subgroup of $O^{3'}(M_i/O_3(M_i))$. We make use of GAP and \cite{maximals} to compute all such choices, which are listed below:
        \begin{table}[H]
            \centering
            \begin{tabular}{lll}
                $A_1 := \Sp_2(3) \wr \Alt(4)$, & $A_2 := \Sp_2(3) \wr 3 \times \SL_2(3)$, & $A_3 := \SL_2(3) \times \Sp_4(3)$, \\
                $A_4 := 2^{1+4}_- \ldotp \Alt(5) \times \Sp_4(3)$, & $A_5 := \SL_2(5) \times \Sp_4(3)$, & $A_6 := (Q_8 \times Q_8) : 3 \times \Sp_4(3)$, \\
                $A_7 := \SL_2(3^3) : 3 \times \SL_2(3)$, & $A_8 := \Alt(4) \times \Sp_4(3)$, & $A_9 := 13 : 3 \times \Sp_4(3)$.
            \end{tabular}
        \end{table}
        
        Assume first that $G$ is isomorphic to a subgroup of $M_j$ for $j \in \{1,3,5\}$. As $|Z(A_i)| \geq 2$ for any $1 \leq i \leq 9$, we see that $G$ is a subgroup of $T := C_{M_j}(t)$ for some involution $t \in O^{3'}(M_j)$. We compute using MAGMA that $|C_V(S_0)| > 3$ for $S_0 \in \Syl_3(T)$ in each case, which is a contradiction.
        
        We next consider whether $G$ can be isomorphic to a subgroup of $M_2$, which satisfies $|Z(O^{3'}(M_2))| = 2$. Then $G$ is isomorphic to a subgroup of $O^{3'}(M_2/O_3(M_2)) \cong \SL_2(3) \times \Sp_6(3)$. This implies that $G$ is isomorphic to one of the following groups -- $A_2$, $A_3$, $A_6$ and $A_7$. In each case, we see that $|Z(A_i)| \geq 4$. We deduce again that $G$ is a subgroup of $T := C_{M_2}(t)$ for some non-central involution $t \in O^{3'}(M_2)$. We compute using MAGMA that $|C_V(S_0)| > 3$ for $S_0 \in \Syl_3(T)$ in each case, contradicting our hypothesis again.
        
        Next assume that $G$ is isomorphic to a subgroup of $M_4$, which also satisfies $|Z(O^{3'}(M_4))| = 2$. Then $G$ is isomorphic to a subgroup of $O^{3'}(M_4/O_3(M_4)) \cong \SL_4(3) \times \Sp_2(3)$. As such, $G$ is isomorphic to $A_3$, meaning that $|Z(G)| \geq 4$. We see again that $G$ is isomorphic to a subgroup of $T := C_{M_4}(t)$ for some non-central involution $t \in O^{3'}(M_4)$. However, we have that $|C_V(S_0)| > 3$ for $S_0 \in \Syl_3(T)$ in each case, which is a contradiction.

        We next rule out $M_7$. In this case, we find that $M_7$ satisfies the first two conditions and that $|C_V(S_0)| = 3$ for $S_0 \in \Syl_3(M_7)$. But we find that for every complement $A$ of a subgroup isomorphic to $3^{1+2}_+$ in such $S_0$, we have that $|C_V(A)| = 3^3$, a contradiction.

        This leaves us with $M_6$. We note that $O^{3'}(M_6) \cong \U_5(2)$ satisfies all the conditions. Also, we appeal to \cite[Tables 8.20 and 8.21]{maximals} to see that every maximal subgroup $M$ of $O^{3'}(M_6)$ satisfies $[M : O_3(M)]_3 < |S|$. As such, we find that $G$ must be isomorphic to $\U_5(2)$. We further compute that $G$ acts irreducibly on $V$.
    \end{proof}

    \begin{lemma} \label{lm:aut-sl7}
        Let $V$ be an elementary abelian group of order $3^7$ and $G \leq \Aut(V)$. Set $S \in \Syl_3(G)$. Assume that:
        \begin{enumerate}
            \item $O^{3'}(G) = G$ and $O_3(G) = 1$; and
            \item $S$ is isomorphic to a Sylow $3$-subgroup of $\SO_7(3)$.
        \end{enumerate}
        Then $G \cong \SO_7(3)$.
    \end{lemma}
    \begin{proof}
        All code used in this proof can be found in \texttt{fi24/sl7.g}. Since $O^{3'}(G) = G$, we know that $G$ is isomorphic to a subgroup of $\SL_7(3)$ and that $|S| = 3^9$. The maximal subgroups of $\SL_7(3)$ are given in \cite[Tables 8.35 and 8.36]{maximals}, which are:
        \begin{table}[H]
            \centering
            \resizebox{\textwidth}{!}{\begin{tabular}{ccc}
                $M_1 := 3^{12} : (\SL_3(3) \times \SL_4(3)) : 2$, & $M_2 := 3^{12} : (\SL_3(3) \times \SL_4(3)) : 2$, &
                $M_3 := 3^{10} : (\SL_2(3) \times \SL_5(3)) : 2$ \\
                $M_4 := 3^{10} : (\SL_2(3) \times \SL_5(3)) : 2$, & $M_5 := 3^6 : \GL_6(3)$, & $M_6 := 3^6 : \GL_6(3)$, \\
                $M_7 := \SO_7(3)$, & & $M_8 := 1093 : 7$
            \end{tabular}}
        \end{table}
        \noindent Clearly, $G$ is not isomorphic to a subgroup of $M_8$. We show that $G$ also cannot be isomorphic to a subgroup of the maximal subgroups $M_i$ for $1 \leq i \leq 6$.

        Assume first that $G$ is isomorphic to a subgroup of $M_1$ (or $M_2$). Since $O_3(G) = 1$ and $O^{3'}(G) = G$, we find that $G$ is isomorphic to a subgroup of $A := \SL_3(3) \times \SL_4(3)$. We compute that a Sylow $3$-subgroup of $A$ has class $3$. But $S$ has nilpotency class $5$, meaning that $S$ cannot be isomorphic to a subgroup of $M_1$ (or $M_2$). Moreover, $G$ cannot be isomorphic to a subgroup of $M_3$ (or $M_4$) as the subgroup $T \in \Syl_3(M_3/O_3(M_3))$ has class $4$. Finally, if $G$ is isomorphic to a subgroup of $M_5$ (or $M_6$), then we appeal to \cite[Tables 8.24 and 8.25]{maximals} to find that $G$ must be isomorphic to a subgroup $A := \Sp_6(3)$. We compute that $T \in \Syl_3(A)$ has order $|S| = 3^9$ but is not isomorphic to $S$. Thus, $G$ cannot be isomorphic to a subgroup of $M_5$ (or $M_6$) either.

        We finally consider $M_7$. The group $M_7 = \SO_7(3)$ satisfying all the hypotheses. Using \cite[Tables 8.39 and 8.40]{maximals}, we see that if $M$ is a maximal subgroup of $M_7$ such that $|M|_3 \geq 3^9$, then $O_3(M) \neq 1$. As such, we are forced to have that $G \cong \SO_7(3)$.
    \end{proof}
    
    \begin{lemma} \label{lm:aut-sp12}
        Let $Q := 3^{1+12}_+$ be the extraspecial group of exponent $3$ and order $3^{13}$, $G \leq \Out(Q)$ and $V := Q/Z(Q)$. Set $S \in \Syl_3(G)$. Assume that:
        \begin{enumerate}
            \item $O^{3'}(G) = G$ and $O_3(G) = 1$; 
            \item $|C_V(S)| = 3$; and
            \item $S$ is isomorphic to a Sylow $3$-subgroup of $2 \ldotp \Suz$.
        \end{enumerate}
        Then $G \cong 2 \ldotp \Suz$, and $G$ acts irreducibly on $V$.
    \end{lemma}
    \begin{proof}
        All code used in this proof can be found in \texttt{m/sp12-subs.m} and \texttt{m/sp12.g}. Since $O^{3'}(G) = G$, we know by \cite{extraspecial-autgrp-sp} that $G$ is isomorphic to a subgroup of $\Sp_{12}(3)$. The maximal subgroups of $\Sp_{12}(3)$ are given in \cite[Tables 8.80 and 8.81]{maximals}. Using MAGMA, we compute those maximal subgroups $M$ such that $|C_V(T)| = 3$ and $[M : O_3(M)]_3 \geq |S|$ for some $T \in \Syl_3(M)$. These are given below:
        \begin{table}[H]
            \centering
            \resizebox{\textwidth}{!}{\begin{tabular}{lll}
                $M_1 := 2 \times 3^{1+10}_+ : \Sp_{10}(3)$, & $M_2 := 3^{3+16} : (\GL_2(3) \times \Sp_8(3))$, & {\color{black} $M_3 := 3^{6+18} : (\GL_3(3) \times \Sp_6(3))$}, \\
                {\color{black} $M_4 := 3^{10+16} : (\GL_4(3) \times \Sp_4(3))$}, & {\color{black} $M_5 := 3^{15+10} : (\GL_5(3) \times \Sp_2(3))$}, & {\color{black} $M_6 := 3^{21} : \GL_6(3)$}, \\
                {\color{black} $M_7 := \Sp_4(3) \wr \Sym(3)$}, & {\color{black} $M_8 := \Sp_4(3^3) : 3$} & {\color{black} $M_9 := \Sp_2(3) \circ \GO_6^+(3)$}, \\
                {\color{black} $M_{10} := \Sp_2(3) \circ \GO_6^-(3)$}, & & {\color{black} $M_{11} := 2 \ldotp \Suz$}.
            \end{tabular}}
        \end{table}
        
        We start by ruling out the maximal subgroups $M_i$ for $1 \leq i \leq 10$. To do so, we iteratively compute relevant subgroups of $M_i$ and show that $G$ is not isomorphic to a subgroup of $M_i$. We only list the base cases here, where every maximal subgroup is invalid (the subgroups $A_i$), or where a Sylow $3$-subgroup has order $|S| = 3^7$, but is not isomorphic to $S$ (the subgroups $B_i$).

        Using \cite{maximals} and MAGMA, we see that the group $A$ given below is such that any maximal subgroup $A_0$ of $A$ such that $|A_0|_3 \geq |S|$ satisfies $O_3(A_0) \neq 1$.
        \begin{table}[H]
            \centering
            \begin{tabular}{c|c}
                Group & Reference \\
                \hline
                $A_1 := \Sp_4(3^2)$ & \cite[Tables 8.12 and 8.13]{maximals} \\
                $A_2 := \SL_5(3)$ & \cite[Tables 8.18 and 8.19]{maximals} \\
                $A_3 := \SU_5(3)$ & \cite[Tables 8.20 and 8.21]{maximals} \\
                $A_4 := \Sp_6(3)$ & \cite[Tables 8.28 and 8.29]{maximals} \\
                $A_5 := \Sp_4(3) \times \Sp_4(3)$ & MAGMA \\
            \end{tabular}
        \end{table}

        The following are some relevant groups with a Sylow $3$-subgroup of order $|S| = 3^7$.
        \begin{table}[H]
            \centering
            \begin{tabular}{lll}
                $B_1 := \Sp_2(3) \circ \SO^+_6(3)$ & $B_2 := \Sp_2(3) \circ \SO^-_6(3)$ & $B_3 := \SL_2(3) \times \SL_4(3)$ \\
                $B_4 := \SL_2(3) \times \SU_4(3)$ & $B_5 := \SL_2(3) \times \SL_2(3^2) \times \Sp_4(3)$ & $B_6 := \SL_2(3)^3 \times \Sp_4(3)$ \\
                $B_7 := \SL_2(3) \wr 3 \times \SL_3(3)$ & $B_8 := (Q_8 \times Q_8) : 3 \times \SL_4(3)$ & $B_9 := \SL_2(5) \times \SL_4(3)$ \\
                $B_{10} := \SL_2(3^3) \times \Sp_4(3)$ & $B_{11} := \SL_3(3) \times \Sp_4(3)$ & $B_{12} := \SU_3(3) \times \Sp_4(3)$ \\
                & $B_{13} := \SL_4(3) \times 2^{1+4}_- \ldotp \Alt(5)$
            \end{tabular}
        \end{table}
        \noindent As $S$ is not indecomposable as a group, we deduce that $G$ is not isomorphic to a subgroup of $B_i$ for $3 \leq i \leq 13$. Moreover, we compute using GAP that $G$ is not isomorphic to a subgroup of $B_1 = O^{3'}(M_{9})$ and $B_2 = O^{3'}(M_{10})$ either. Using the groups $A_i$ and $B_j$, along with \cite{maximals} and MAGMA, we can show that $G$ is not isomorphic to a subgroup of $M_i$ for $1 \leq i \leq 10$ (see \texttt{m/sp-12.m}). 
        
        We further note that $M_{11} = O^{3'}(M_{11})$ satisfies all the assumptions listed above. We also compute that $M_{11}$ acts irreducibly on $V$. We appeal to the ATLAS \cite{atlas} to find that every maximal subgroup $D$ of $M_{11}$ containing a Sylow $3$-subgroup of $M_{11}$ satisfies $O_3(D) \neq 1$. Thus, $G$ cannot be isomorphic to a proper subgroup of $M_{11}$. We conclude that $G \cong 2 \ldotp \Suz$.
    \end{proof}
    
    \section{Fusion Systems on a Sylow $3$-subgroup of $\Fi_{24}'$} \label{sec:fi24}
    In this section, we classify all corefree fusion systems on a Sylow $3$-subgroup of $\Fi_{24}'$. From the ATLAS \cite{atlas}, we find that the following $3$-local subgroups are maximal in $\Fi_{24}'$:
    \begin{align*}
        M_1 &\cong 3^{1+10}_+ : (\U_5(2) : 2) \\
        M_2 &\cong 3^{2+4+8} : (\Alt(5) \times \SL_2(3)) : 2 \\
        M_3 &\cong 3^{3+4+3+3} : \GL_3(3) \\
        M_4 &= 3^7 \ldotp \SO_7(3).
    \end{align*}
    
    Fix a Sylow $3$-subgroup $S$ of $G := \Fi_{24}'$. We can choose a $G$-conjugate of $M_i$ so that $S \in \Syl_3(M_i)$ for $1 \leq i \leq 4$. We now construct further $3$-local subgroups of $G$:
    \begin{align*}
        N_1 := M_1 \cap M_3 &\cong 3^{1+10}_+ : (3 \times 3^{1+2}_+) : (2 \times \GL_2(3)) \\
        N_2 = M_1 \cap M_2 &\cong 3^{1+10}_+ : 3^4 : (2 \times \Sym(5)) \\
        N_3 = M_2 \cap M_3 &\cong 3^{2+4+8} : 3 : (2 \times \GL_2(3)) 
    \end{align*}

    \begin{notation}
        Using the $3$-local subgroups of $G$ given above, we can define the following subgroups of $S$:
        \begin{enumerate}
            \item $\bfW := O_3(M_1)$;
            \item $\bfU := O_3(M_2)$;
            \item $\bfT := O_3(M_3)$;
            \item $\bfV_1 := O_3(M_4)$;
            \item $P := O_3(N_1)$;
            \item $Q := O_3(N_2)$; and
            \item $R := O_3(N_3)$.
        \end{enumerate}
    \end{notation}
    The group $\bfV_1$ has a second $N_G(S)$-conjugate that we shall label $\bfV_2$. These subgroups can also be characterized based on the structure of $S$:
    \begin{enumerate}
        \item $\bfW$ is the preimage of $J(S/Z(S))$ in $S$;
        \item $\bfV_1$ and $\bfV_2$ are the two elementary abelian subgroups in $S$ of maximal order;
        \item $\bfT = C_S(\gamma_7(S))$;
        \item $P = C_S(Z_2(S/\bfV_1)) = C_S(Z_2(S/\bfV_2))$;
        \item $Q = C_S(Z_2(S/\bfW))$;
        \item $R = C_S(Z_2(S))$; and
        \item $\bfU = Q \cap R$. 
    \end{enumerate}
    Except for $\bfV_i$, we see that every subgroup given above is characteristic in $S$. Moreover, the subgroups $\bfV_i$ are normal in $S$.
    
    We set $H_i := N_G(\bfV_i) \cong M_4$. We find that 
    \begin{align*}
        O^{3'}(\Out_{H_i}(P)) &= O^{3'}(\Out_G(P)) \cong \SL_2(3), \\
        \Alt(4) \cong O^{3'}(\Out_{H_i}(Q)) &< O^{3'}(\Out_G(Q)) \cong \Alt(5), \\
        O^{3'}(\Out_{H_i}(R)) &= O^{3'}(\Out_G(R)) \cong \SL_2(3).
    \end{align*}
    We note that $P/\bfV_i$, $Q/\bfV_i$ and $R/\bfV_i$ are precisely the unipotent radicals of the minimal parabolics in $H_i/\bfV_i \cong \SO_7(3)$. We further have that
    \[H_1 \cap H_2 = O^{3'}(N_G(\bfT)) \cong 3^{3+4+3+3} : \SL_3(3).\]

    \begin{lemma}
        Let $\calF := \calF_S(\Fi_{24}')$. Then $\calE(\calF) = \{P, Q, R\}$ and $\calF^{cr} = \calE(\calF) \cup \{S, \bfT, \bfU, \bfW, \bfV_1^{\Aut(S)}\}$.
    \end{lemma}
    \begin{proof}
        This follows from \cite{fi24-radicals}.
    \end{proof}

    We now let $\calF$ be a saturated fusion system on $S$. We can construct $G = \Fi_{24}'$ in GAP, and work with an efficient representation of $S$. We also construct the quotient $\Out_G(\bfW)$ in MAGMA by finding the relevant maximal subgroup in $O^{3'}(\Out(\bfW)) \cong \Sp_{10}(3)$. Throughout this section, we refer the reader less familiar with the structure of $S$ to different pieces of code in MAGMA and GAP to verify the relevant assertions.

    \begin{lemma} 
        We have $\calE(\calF) \subseteq \{P, Q, R\}$.
    \end{lemma}
    \begin{proof}
        See Appendix \ref{sec:algorithm-appendix}.
    \end{proof}

    The code relevant to Lemmas \ref{lm:fi24-w} to \ref{lm:fi24-t} can be found in the file \texttt{fi24/weak-closure.g}.
    
    \begin{lemma} \label{lm:fi24-w}
        The subgroup $\bfW$ is weakly closed in $\calF$. In particular, $\calE(N_\calF(\bfW)) = \calE(\calF) \cap \{P,Q\}$.
    \end{lemma}
    \begin{proof}
        We note that $P$ and $Q$ contain $\bfW$. Since $Z(\bfW) = Z(S)$, $\bfW$ is not a subgroup of $R = C_S(Z_2(S))$. As $Z(S) = Z(P) = Z(Q)$ and $\bfW/Z(S) = J(S/Z(S))$, we deduce by Lemma \ref{lm:grp-thompson} that $\bfW/Z(Q) = J(Q/Z(Q))$ and $\bfW/Z(P) = J(P/Z(P))$. We infer that $\bfW$ is normalized by $\Aut_\calF(P)$, $\Aut_\calF(Q)$ and $\Aut_\calF(S)$. By Alperin-Goldschmidt, we deduce that $\bfW$ is weakly closed in $\calF$. The second part follows from Lemma \ref{lm:essentials-in-normalizer}.
    \end{proof}

    \begin{lemma} \label{lm:fi24-u}
        The subgroup $\bfU$ is weakly closed in $\calF$. In particular, $\calE(N_\calF(\bfU)) = \calE(\calF) \cap \{Q, R\}$.
    \end{lemma}
    \begin{proof}
        By construction, we have $\bfU = Q \cap R$. Since $Q$ and $R$ are $\Aut_\calF(S)$-invariant, it follows that $\bfU$ is $\Aut_\calF(S)$-invariant. We also find that $\bfU = C_Q(Z_2(Q))$ and that $\bfU = C_R(Z_4(R)/Z_2(R))$. As such, $\bfU$ is normalized by $\Aut_\calF(A)$ for every $A \in \calE(\calF) \cup \{S\}$ containing $\bfU$. Therefore, $\bfU$ must be weakly closed in $\calF$. The second part follows from Lemma \ref{lm:essentials-in-normalizer}.
    \end{proof}

    \begin{lemma} \label{lm:fi24-t}
        The subgroup $\bfT$ is weakly closed in $\calF$. In particular, $\calE(N_\calF(\bfT)) = \calE(\calF) \cap \{P, R\}$.
    \end{lemma}
    \begin{proof}
        By construction, we have that $\bfT \leq P \cap R$, with $S/\bfT \cong 3^{1+2}_+$. Moreover, $\bfT$ is not contained in $Q$. As $\gamma_7(S) = \gamma_6(R) = \gamma_5(P)$, we infer that
        \[\bfT = C_S(\gamma_7(S)) = C_R(\gamma_6(R)) = C_P(\gamma_5(P)).\]
        As such, $\bfT$ is normalized by $\Aut_\calF(S)$, $\Aut_\calF(P)$ and $\Aut_\calF(R)$. This implies that $\bfT$ is weakly closed in $\calF$. The second part follows from Lemma \ref{lm:essentials-in-normalizer}.
    \end{proof}

    The code relevant to Lemmas \ref{lm:fi24-outfp} and \ref{lm:fi24-outfr} are given in \texttt{fi24/automizer.g}.

    \begin{lemma} \label{lm:fi24-outfp}
        If $P \in \calE(\calF)$, then $P/W_0$, $P/\bfT$ and $Z_2(P)/Z(P)$ are natural modules for $O^{3'}(\Out_\calF(P)) \cong \SL_2(3)$, where $W_0 := C_P(P/\bfW)$.
    \end{lemma}
    \begin{proof}
        Assume that $r \in O^{3'}(\Aut_\calF(P))$ is a $3'$-element that centralizes $P/W_0$. Then we find that $[r, P, W_0] \leq [W_0, W_0] \leq Z_4(P)$. Moreover, since $[[P, W_0] : Z_4(P) \cap [P, W_0]] = 3$, we also have that $[P, W_0, r] \leq Z_4(P)$. By the Three Subgroups Lemma, we deduce that $[W_0, r, P] \leq Z_4(P)$. As such, we find that $r$ centralizes $W_0/Z_5(P)$. Since $[Z_5(P) : \Phi(P)] = 3$, we find that $O^{3'}(\Aut_\calF(P))$ centralizes $Z_5(P)/\Phi(P)$. By coprime action, we deduce that $r$ centralizes $W_0/\Phi(P)$. But then $r$ centralizes $P/\Phi(P)$, meaning that $r = 1$. Thus, $O^{3'}(\Out_\calF(P))$ acts faithfully on $P/W_0$ of order $3^2$. This means that $P/W_0$  is a natural module for $O^{3'}(\Out_\calF(P)) \cong \SL_2(3)$.

        Now let $r \in O^{3'}(\Aut_\calF(P))$ be a $3'$-element that centralizes $Z_2(P)/Z(P)$. Since $|Z(P)| = [\gamma_4(P) : Z_2(P)] = 3$, we deduce by coprime action that $r$ centralizes $\gamma_4(P)$. This implies that $[r, \gamma_4(P), P] = 1 = [\gamma_4(P), P, r]$. By the Three Subgroups Lemma, we deduce that $[P, r, \gamma_4(P)] = 1$. Thus, $r$ centralizes $P/C_P(\gamma_4(P))$. Since $[C_P(\gamma_4(P)) \Phi(P) : C_P(\gamma_4(P)] = 3$, we deduce by coprime action that $r$ centralizes $(C_P(\gamma_4(P)) \Phi(P))/\Phi(P)$. But then $r$ centralizes $P/\Phi(P)$, meaning that $r = 1$. Thus, $Z_2(P)/Z(P)$, which has order $3^2$, is a natural module for $O^{3'}(\Out_\calF(P)) \cong \SL_2(3)$.

        Finally, let $r \in O^{3'}(\Aut_\calF(P))$ be a $3'$-element that centralizes $P/\bfT$ of order $3^2$. As $Z(\bfT) = Z_2(P)$, we infer that $[r, P, Z(\bfT)] = 1 = [P, Z(\bfT), r]$. By the Three Subgroups Lemma, we deduce that $[Z(\bfT), r, P] = 1$, meaning that $[Z(\bfT), r] \leq Z(P)$. In particular, $r$ centralizes $Z(\bfT)/Z(P) = Z_2(P)/Z(P)$. By the previous paragraph, it follows that $r = 1$. Thus, $P/\bfT$ is also a natural module for $O^{3'}(\Out_\calF(P)) \cong \SL_2(3)$.
    \end{proof}

        

        

    \begin{lemma} \label{lm:fi24-outfr}
        If $R \in \calE(\calF)$, then both $R/\bfT$ and $Z(R)$ are natural modules for $O^{3'}(\Out_\calF(R)) \cong \SL_2(3)$.
    \end{lemma}
    \begin{proof}
        Let $r \in O^{3'}(\Aut_\calF(R))$ be a $3'$-element that centralizes $R/\bfT$. Since $[\bfT : Z_6(R)] = 3$, we deduce by coprime action that $r$ centralizes $\bfT/Z_6(R)$. This implies that $[r, R, R] = [R, r, R] \leq Z_5(R)$, meaning that $[R, R, r] \leq Z_5(R)$ by the Three Subgroups Lemma. In other words, $r$ centralizes $\Phi(R)/(\Phi(R) \cap Z_5(R))$. Thus, $r$ also centralizes $\Phi(R) Z_5(R)/\Phi(R)$. We compute that $\Phi(R) Z_5(R) = Z_6(R)$. By coprime action, it follows that $r$ centralizes $R/\Phi(R)$. We conclude that $r = 1$. Thus, $O^{3'}(\Out_\calF(R))$ acts faithfully on $R/\bfT$ of order $3^2$. This means that $R/\bfT$ is a natural module for $O^{3'}(\Out_\calF(R)) \cong \SL_2(3)$.
        
        Let $r \in O^{3'}(\Aut_\calF(R))$ be a $3'$-element that centralizes $Z(R)$ of order $3^2$. As $Z_2(R) = Z_3(S)$ and $Z(R) = Z_2(S)$, we deduce that $r$ centralizes $Z_2(R)$ as well by coprime action. Let $B_i := Z_i(R) \cap \Phi(R)$ for $2 \leq i \leq 4$. We compute that $[B_4 : B_3] = [B_3 : B_2] = 3$. As such, $r$ centralizes $B_4/B_2$. However, we have $B_2 = Z_2(R)$, which implies that $r$ centralizes $B_2$ as well. But then $r$ centralizes $B_4$ by coprime action. As $B_4$ is self-centralizing in $R$, we infer that $r = 1$. This implies that $Z(R)$ is a natural $\SL_2(3)$-module for $O^{3'}(\Out_\calF(R))$.
    \end{proof}
    
    \begin{proposition} \label{prp:fi24-outfw}
        Assume that $O_3(N_\calF(\bfW)) = \bfW$. Then $O^{3'}(\Out_\calF(\bfW)) \cong \U_5(2)$ acts irreducibly on $\bfW/Z(S)$. Moreover, $\calE(N_\calF(\bfW)) = \{P, Q\}$, $O^{3'}(\Out_\calF(Q)) \cong \Alt(5)$ and $|\Out_{O^{3'}(N_\calF(\bfW))}(S)| = 2^2$.
    \end{proposition}
    \begin{proof}
        The first part follows from Lemma \ref{lm:aut-sp10}. The second statement follows by analysing the structure of $\Out_\calF(\bfW)$ (see, for example, the file \texttt{fi24/automizer-w.m}).
    \end{proof}
    
    The code for Lemmas \ref{lm:fi24-vnorm-p} and \ref{lm:fi24-vnorm-r} can be found in the file \texttt{fi24/v-norm.g}.
    
    \begin{lemma} \label{lm:fi24-vnorm-p}
        If $P \in \calE(\calF)$, then both $\bfV_i$ are normalized by $O^{3'}(\Aut_\calF(P))$.
    \end{lemma}
    \begin{proof}
        Let $W_0 := C_P(P/\bfW) = C_S(P/\bfW)$. Then $O^{3'}(\Aut_\calF(P))$ centralizes $W_0/\bfW$ of order $3^2$. In particular, the four intermediate subgroups of $\bfW$ and $W_0$, which we label $C_i$ for $1 \leq i \leq 4$, are $O^{3'}(\Aut_\calF(P))$-invariant. We compute that, up to relabelling, $C_R(Z_2(C_1))$ and $C_R(Z_2(C_2))$ are the two $\Aut(S)$-conjugates of $\bfV_1$. Thus, both $\bfV_i$ are normalized by $O^{3'}(\Aut_\calF(P))$.
    \end{proof}
    
    \begin{lemma} \label{lm:fi24-vnorm-r}
        If $R \in \calE(\calF)$, then both $\bfV_i$ are normalized by $O^{3'}(\Aut_\calF(R))$.
    \end{lemma}
    \begin{proof}
        We compute that $\gamma_5(R) \leq \Phi(\Phi(R)) \leq Z_2(\Phi(R))$, with $|\Phi(\Phi(R)) : \gamma_5(R)| = |Z_2(\Phi(R)) : \Phi(\Phi(R))| = 3$. We infer that $O^{3'}(\Aut_\calF(R))$ centralizes $Z_2(\Phi(R))/\gamma_5(R)$. In particular, the four intermediate subgroups of $\gamma_5(R)$ and $Z_2(\Phi(R))$, which we label $C_i$ for $1 \leq i \leq 4$, are $O^{3'}(\Aut_\calF(R))$ invariant. We again appeal to GAP to compute that, up to relabelling, $J(C_R(Z_2(C_1)))$ and $J(C_R(Z_2(C_2)))$ are the two $\Aut(S)$-conjugates of $\bfV_1$.
    \end{proof}

    \begin{lemma} \label{lm:fi24-outfq}
        Assume that $\{P, Q\} \subseteq \calE(\calF)$. Then one of the following holds:
        \begin{enumerate}
            \item $O^{3'}(\Out_\calF(Q)) \cong \Alt(4)$, $A := O_3(N_\calF(\bfW))$ satisfies $[\bfW : A] = 3$, and $\bfV_i$ is normal in $N_\calF(\bfW)$ (for a unique $i \in \{1,2\}$), or
            \item $O^{3'}(\Out_\calF(Q)) \cong \Alt(5)$, and $O_3(N_\calF(\bfW)) = \bfW$.
        \end{enumerate}
    \end{lemma}
    \begin{proof}
        All code used in this proof can be found in the file \texttt{fi24/automizer-q.g}. If $O_3(N_\calF(\bfW)) = \bfW$, then we know by Proposition \ref{prp:fi24-outfw} that $O^{3'}(\Out_\calF(Q)) \cong \Alt(5)$.

        We now assume that $\bfW < A = O_3(N_\calF(\bfW))$. Since $\{P, Q\} \subseteq \calE(\calF)$, we know that $O_3(N_\calF(\bfW)) \leq P \cap Q$ by \cite[Proposition I.4.5]{ako}. Let $W_0 := C_S(S/\bfW)$. We know by Lemma \ref{lm:fi24-outfp} that $O^{3'}(\Out_\calF(P))$ acts irreducibly on $P/W_0$. This means that $A \leq W_0$. 
        
        We know that $\bfW$ is weakly closed in $\calF$ by Lemma \ref{lm:fi24-w}. This implies that $A$ is normalized by $\Aut_\calF(Q)$. Since $A \leq W_0$, we deduce that $O^{3'}(\Out_\calF(Q))$ acts trivially on $A/\bfW$. We compute that $\bfW/\Phi(Q)$ has order $3$, meaning that $O^{3'}(\Out_\calF(Q))$ acts trivially on $\bfW/\Phi(Q)$ as well. By coprime action, $O^{3'}(\Out_\calF(Q))$ must then act faithfully on $Q/A$. Since $\bfW < A$, we further know that $[Q : A] \leq 3^3$. This implies that $O^{3'}(\Out_\calF(Q))$ is a subgroup of $\SL_3(3)$. We appeal to \cite[Tables 8.3 and 8.4]{maximals} to find that there are three choices for $O^{3'}(\Out_\calF(Q))$ -- $\SL_2(3)$, $\Alt(4)$ and $13 : 3$.

        If $O^{3'}(\Out_\calF(Q)) \cong \SL_2(3)$, then $Q/\bfW$ has a unique non-central chief factor for the action of $O^{3'}(\Out_\calF(Q))$, which is a natural $\SL_2(3)$-module. Set $X := O_2(O^{3'}(\Out_\calF(Q)))$. By coprime action, we have
        \[Q/\bfW = [Q/\bfW, X] \times C_{Q/\bfW}(X),\]
        with $A/\bfW \leq C_{Q/\bfW}(X)$. As $[Q/\bfW, X]$ is a natural $\SL_2(3)$-module, we infer both direct factors given above have order $3^2$. This implies that $[Q, \Out_S(Q), \Out_S(Q)] \leq \bfW$. But we calculate that $[Q, S, S] \nleq \bfW$, which is a contradiction. 
        
        Suppose now that $O^{3'}(\Out_\calF(Q)) \cong 13 : 3$. By Lemma \ref{lm:fi24-outfp}, we know that $O^{3'}(\Out_\calF(P)) \cong \SL_2(3)$. By the extension axiom, there exists an involution $t \in \Aut_\calF(S)$ such that $t|_P \in N_{O^{3'}(\Aut_\calF(P))}(\Aut_S(P))$. In Lemma \ref{lm:fi24-outfp}, we saw that $t$ acts as $1$ on $W_0/\bfW$, but as $-1$ on $(P \cap Q)/W_0$. Moreover, we have $A \leq W_0 \leq P \cap Q$, meaning that $t$ acts non-trivially on $Q/A$. On the other hand, since $[A : W_0] = 3$, we know that $t$ cannot act fixed point freely on $Q/A$. Let $B := C_{Q/\bfW}(O^{3'}(\Out_\calF(Q)))$. Then $B$ is an $O^{3'}(\Aut_\calF(Q))$-invariant subgroup of $Q$ such that $[B : \bfW] = 3$. This forces $B = A$. In particular, $O^{3'}(\Out_\calF(Q))$ acts irreducibly on $Q/A$. Inside $D := \GL_3(3)$, we compute using GAP that that any subgroup $T$ isomorphic to $13 : 3$ satisfies $N_D(T) = \langle T, Z(D) \rangle$. This means that $t$ must correspond to the central involution in $D$. But then $t$ acts fixed point freely on $Q/A$, a contradiction.

        We have shown that if $\bfW < A = O_3(N_\calF(\bfW))$, then $O^{3'}(\Out_\calF(Q)) \cong \Alt(4)$. As the smallest faithful $\GF(3)$-module of $\Alt(4)$ has dimension $3$, we infer that $[A : \bfW] = 3$. We recall that $\bfW \leq A \leq W_0$. Since $W_0/\bfW$ is elementary abelian of order $3^2$, there are precisely four choices for $A$. We show that only two of the four choices are valid. To see this, consider the subgroup $\hat{C} := \langle A, C_Q(\gamma_3(A)) \rangle$. As $A$ is invariant under $O^{3'}(\Aut_\calF(Q))$, $\hat{C}$ must also be invariant under $O^{3'}(\Aut_\calF(Q))$. We compute that $\hat{C}$ properly contains $A$ in each case, so it must equal $Q$. But this only holds for two choices for $A$, which we label $A_1$ and $A_2$. We observe that $A_1$ and $A_2$ are $\Aut(S)$-conjugate. Then, up to relabelling, we have that $\bfV_i = C_{A_i}(Z_2(A_i))$. In particular, as some $A_i$ is normal in $N_\calF(\bfW)$, we deduce that some $\bfV_i$ is also normal in $N_\calF(\bfW)$.

        Assume now that both $\bfV_i$ are normal in $N_\calF(\bfW)$. Then $T := \langle J(S), \Phi(Q) \rangle$ must also be normal in $N_\calF(Q)$. As $[Q : T] = 3^3$ and $[T : \Phi(Q)] = 3^2$, we infer that $O^{3'}(\Aut_\calF(Q)) \cong \Alt(4)$ must act irreducibly on $Q/T$. But we find that $T < \langle T, \bfW \rangle < Q$ is also normal in $N_\calF(Q)$, a clear contradiction. Thus, exactly one of $\bfV_1$ or $\bfV_2$ is normal in $N_\calF(\bfW)$.
    \end{proof}

    \begin{lemma} \label{lm:fi24-outfu}
        Assume that $\{Q, R\} \subseteq \calE(\calF)$.
        \begin{enumerate}
            \item If $O^{3'}(\Out_\calF(Q)) \cong \Alt(4)$, then $O^{3'}(\Out_\calF(\bfU)) \cong \SL_2(3) \times \Alt(4)$;
            \item If $O^{3'}(\Out_\calF(Q)) \cong \Alt(5)$, then $O^{3'}(\Out_\calF(\bfU)) \cong \SL_2(3) \times \Alt(5)$.
        \end{enumerate}
    \end{lemma}
    \begin{proof}
        We know by Lemma \ref{lm:fi24-u} that $\Aut_\calF(Q)$ and $\Aut_\calF(R)$ normalize $\bfU$. Since $Q, R \in \calE(\calF)$, we know by Lemma \ref{lm:weak-closure-to-normality} that 
        \[\bfU \leq O_3(N_\calF(\bfU)) \leq Q \cap R = \bfU.\]
        In particular, $O_3(\Out_\calF(\bfU)) = 1$.
        
        We recall that $\bfU$ has shape $3^{2+4+8}$, with $\Phi(\bfU)$ elementary abelian subgroup of order $3^6$. Moreover, $\Phi(\bfU) = Z_2(\bfU)$ is self-centralizing in $\bfU$. As such, $H := O^{3'}(\Out_\calF(\bfU))$ acts faithfully on $\Phi(\bfU)$. Consider now the two subgroups $H_1 := C_H(Z(\bfU))$ and $H_2 := C_H(Z_2(\bfU)/Z(\bfU))$. By coprime action, if $r \in H_1 \cap H_2$ is a $3'$-element, then $r$ centralizes $\Phi(\bfU)$, meaning that $r = 1$. This implies that $H_1 \cap H_2 \leq O_3(\Out_\calF(\bfU)) = 1$. 
        
        We compute that $C_S(Z(\bfU)) = R$ and $C_S(Z_2(\bfU)/Z(\bfU)) = Q$. Since $|H|_{3} = 3^2$, we infer that $|H_1|_{3} = |H_2|_{3} = 3$. Moreover, $H_1H_2$ is a normal subgroup of $H$ containing $\Out_S(\bfU)$. As $O^{3'}(H) = H$ and $\Out_S(\bfU)$ is a Sylow $3$-subgroup of $\Out_\calF(\bfU)$, we deduce that $H = H_1 \times H_2$.
        
        Let $r \in \Aut_\calF(\bfU)$ be a $3'$-element that centralizes $Z_2(\bfU)/Z(\bfU)$. By coprime action, we infer that $r$ must act faithfully on $Z(\bfU)$ as $Z_2(\bfU)$ is self-centralizing in $\bfU$. This implies that $H_2$ acts faithfully on $Z(\bfU)$ of order $3^2$. Since $O_3(H_2) = 1$ and $|H_2|_3 = 3$, we deduce that $H_2 \cong \SL_2(3)$.
    
        Consider now the restriction map $\pi \colon \Aut_\calF(Q) \to \Aut_\calF(\bfU)$. Since $\bfU$ is self-centralizing in $Q$, we know that $\ker(\pi)$ is a $3$-group. Since $Z(S) = Z(Q)$ and $Z_2(S) = Z_2(Q)$, we know that any $3'$-element $r \in O^{3'}(\Aut_\calF(Q))$ centralizes $Z_2(Q) = Z(\bfU)$. This implies that $O^{3'}(\Aut_\calF(Q))\pi/\Inn(\bfU) \leq \langle \Out_S(\bfU), H_1 \rangle$.
    
        Since $\Out_Q(\bfU) \leq H_2$, we deduce that $H_1$ centralizes $\Out_Q(\bfU)$. As such, any map $\alpha \in \Aut_\calF(\bfU)$ that centralizes $Z(\bfU)$ normalizes $\Aut_Q(\bfU)$. By the extension axiom, we deduce that $\alpha$ has an extension $\hat{\alpha} \in \Aut_\calF(Q)$. This implies that $\langle H_1, \Out_S(\bfU) \rangle \leq O^{3'}(\Aut_\calF(Q))\pi/\Inn(\bfU)$. We conclude that $H_1 \cong O^{3'}(\Out_\calF(Q))$, as desired.
    \end{proof}

    \begin{lemma} \label{lm:fi24-outfw-alt4}
        Assume that $\{P, Q\} \subseteq \calE(\calF)$, with $O^{3'}(\Out_\calF(Q)) \cong \Alt(4)$. Set $A := O_3(N_\calF(\bfW))$. Then $O^{3'}(\Out_\calF(A)) \cong \SO_5(3)$ acts irreducibly on $A/\bfV$, where $\bfV := C_A(Z_2(A))$.
    \end{lemma}
    \begin{proof}
        All code used in this proof can be found in the file \texttt{fi24/automizer-a.m}. As $A = O_3(N_\calF(\bfW))$, we infer that 
        \[Z(S) \leq Z(A) \leq Z(\bfW) = Z(S).\]
        In particular, $R$ does not contain $A$. By Lemmas \ref{lm:weak-closure-to-normality} and \ref{lm:fi24-w}, we infer that $N_\calF(A) = N_\calF(\bfW)$. This implies that $O_3(N_\calF(A)) = A$, so that $O_3(\Out_\calF(A)) = 1$.
        
        We recall that $\bfV = C_A(Z_2(A))$ is an $\Aut(S)$-conjugate of $\bfV_i$ by Lemma \ref{lm:fi24-outfq}. We compute that $[A : \Phi(A)] = 3^6$ and that $A/\bfV$ is an elementary abelian group of order $3^5$. As such, $O^{3'}(\Out_\calF(A))$ acts faithfully on $A/\bfV$ of order $3^5$. This implies that $O^{3'}(\Out_\calF(A))$ is isomorphic to a subgroup of $\SL_5(3)$. Since $\Out_S(A) \cong 3 \wr 3$, we compute using MAGMA that $O^{3'}(\Out_\calF(A))$ is isomorphic to either $\SO_5(3)$ or $\Sp_4(3)$. As $O^{3'}(\Out_\calF(Q)) \cong \Alt(4)$, we are forced to have that $O^{3'}(\Out_\calF(A)) \cong \SO_5(3)$. Since $O^{3'}(\Out_\calF(A))$ acts faithfully on $A/\bfV$, and $\SO_5(3)$ has no faithful module of dimension less than $5$, it follows that $O^{3'}(\Out_\calF(A))$ acts irreducibly on $A/\bfV$.
    \end{proof}

    \begin{proposition} \label{prp:fi24-outfv}
        Assume that $\calE(\calF) = \{P, Q, R\}$, and let $\calG_i := N_\calF(\bfV_i)$ for $1 \leq i \leq 2$. If $Q \in \calE(\calG_i)$, then $O^{3'}(\Out_\calF(\bfV_i)) = O^{3'}(\Out_{\calG_i}(\bfV_i)) \cong \SO_7(3)$.
    \end{proposition}
    \begin{proof}
        We know by Lemmas \ref{lm:fi24-vnorm-p} and \ref{lm:fi24-vnorm-r} that $P, Q \in \calE(\calG_i)$ for each $i$. As $Q \in \calE(\calG_i)$ and $\bfV_i$ is not isomorphic to a subgroup of $\bfW$, we deduce by Lemma \ref{lm:fi24-outfq} that $O^{3'}(\Out_{\calG_i}(Q)) \cong \Alt(4)$. Set $A_i := O_3(N_{\calG_i}(\bfW))$. As only one $\Aut(S)$-conjugate of $\bfV_i$ can be normal in $N_{\calG_i}(\bfW)$, we infer that $\bfV_i = C_A(Z_2(A))$ (up to relabelling) by the proof of Lemma \ref{lm:fi24-outfq}. 
        
        Let $Y := O_3(\calG_i)$, which contains $\bfV_i$. We know that $O^{3'}(\Out_{\calG_i}(Q))$ acts irreducibly on $Q/A$, meaning that $\bfV_i \leq Y \leq A \cap R < A$. Moreover, Lemma \ref{lm:fi24-outfw-alt4} tells us that $O^{3'}(\Out_\calF(A))$ acts irreducibly on $A/\bfV_i$. This allows us to conclude that $Y = \bfV_i$. As such, we find that $O_3(\Out_\calF(\bfV_i)) = O_3(\Out_{\calG_i}(\bfV_i)) = 1$. We now apply Lemma \ref{lm:aut-sl7} to conclude that $O^{3'}(\Out_{\calG_i}(\bfV_i)) \cong \SO_7(3)$, as desired.
    \end{proof}
    
    \begin{proposition} \label{prp:fi24-trivcore}
        Let $\calF$ be a saturated fusion system on $S$. Then either $\calF$ is constrained or we have $O_3(\calF) = 1$. Moreover, $O_3(\calF) = 1$ if and only if $\calE(\calF) = \{P, Q, R\}$ and $O^{3'}(\Out_\calF(Q)) \cong \Alt(5)$.
    \end{proposition}
    \begin{proof}
        If $R \not\in \calE(\calF)$, then we know by Lemma \ref{lm:fi24-w} that $\bfW \normalIn \calF$. If $P \not\in \calE(\calF)$, then Lemma \ref{lm:fi24-u} tells us that $\bfU \normalIn \calF$. Also, if $Q \not\in \calE(\calF)$, then by Lemma \ref{lm:fi24-t}, we have that $\bfT \normalIn \calF$. Assume now that $O^{3'}(\Out_\calF(Q)) \cong \Alt(4)$. By Lemma \ref{lm:fi24-outfq}, we know that some $\bfV_i$ is normal in $N_\calF(\bfW)$. We further know that $O^{3'}(\Aut_\calF(R))$ normalizes $\bfV_i$ by Lemma \ref{lm:fi24-vnorm-r}. By the Frattini Argument and Lemma \ref{lm:autfe-decomposition}, we infer that $\calF = \langle N_\calF(\bfW), O^{3'}(\Aut_\calF(R)) \rangle$. This implies that $\bfV_i \normalIn \calF$. As $\bfW, \bfU, \bfT$ and $\bfV_i$ are all self-centralizing in $S$, we deduce that either $\calF$ is constrained or $\calE(\calF) = \{P, Q, R\}$ with $O^{3'}(\Out_\calF(Q)) \cong \Alt(5)$. If $O_3(\calF) = 1$, then $\calF$ cannot be constrained. Thus, if $O_3(\calF) = 1$, then we must have $\calE(\calF) = \{P, Q, R\}$ with $O^{3'}(\Out_\calF(Q)) \cong \Alt(5)$.
    
        Now assume that $\calE(\calF) = \{P, Q, R\}$ and $O^{3'}(\Out_\calF(Q)) \cong \Alt(5)$. Then by Lemma \ref{lm:fi24-outfq}, we deduce that $O_3(\calF) \leq O_3(N_\calF(\bfW)) = \bfW$. By Proposition \ref{prp:fi24-outfw}, we have that $O^{3'}(\Out_\calF(\bfW)) \cong \U_5(2)$ acts irreducibly on $\bfW/Z(S)$. Also, we know by Lemma \ref{lm:fi24-outfr} that $O^{3'}(\Out_\calF(R))$ acts irreducibly on $Z(R)$, which properly contains $Z(S)$. As $O_3(\calF) \leq R \cap \bfW < \bfW$, we conclude that $O_3(\calF) = 1$.
    \end{proof}

    We recall that $G = \Fi_{24}'$. Set $\calH$ to be the fusion system on $S$ realized by $\Aut(G)$, and $\calH_0 := O^{3'}(\calH)$ to be the one realized by $G$. Then $\Out_{\calH_0}(S)$ and $\Out_{\calH}(S)$ are both elementary abelian groups of order $2^3$ and $2^4$ respectively.
    \begin{lemma} \label{lm:fi24-autfs}
        Assume that $O_3(\calF) = 1$. Then $\Aut_\calF(S)$ is $\Aut(S)$-conjugate to either $\Aut_\calH(S)$ or $\Aut_{\calH_0}(S)$.
    \end{lemma}
    \begin{proof}
        All code used in this proof can be found in the file \texttt{fi24/autfs.g}. We compute that $\Out_{\calH}(S)$ is a Sylow $2$-subgroup of $\Out(S)$. As such, if $|\Out_\calF(S)| = 2^4$, then $\Aut_\calF(S)$ and $\Aut_\calH(S)$ are $\Aut(S)$-conjugate by Sylow's Theorems.

        Assume now that $|\Out_\calF(S)| \leq 2^3$. As $O_3(\calF) = 1$, we know by Proposition \ref{prp:fi24-trivcore} that $\calE(\calF) = \{P, Q, R\}$ with $O^{3'}(\Out_\calF(Q)) \cong \Alt(5)$. By the extension axiom, there exists an involution $t_1 \in \Aut_\calF(S)$ such that $t_1|_{Q} \in O^{3'}(\Aut_\calF(Q))$. As $O^{3'}(\Out_\calF(R)) \cong \SL_2(3)$, we can find another involution $t_2 \in \Aut_\calF(S)$ such that $t_2|_R \in O^{3'}(\Aut_\calF(R))$. As $\bfU = Q \cap R$ is weakly closed in $\calF$, we know that $t_1$ centralizes $Q/\bfU$ and $t_2$ centralizes $R/\bfU$. As $QR = S$, we infer that $t_1$ centralizes $S/R$ and $t_2$ centralizes $S/Q$. We further know by Lemma \ref{lm:fi24-outfu} that $|\Out_\calF(S)|_{3'} \geq 2^2$.
        
        As $O^{3'}(\Out_\calF(P)) \cong \SL_2(3)$, we know by the extension axiom that there exists a $t_3 \in \Aut_\calF(S)$ such that $t_3|_P \in O^{3'}(\Aut_\calF(P))$. In Lemma \ref{lm:fi24-outfp}, we showed that $P/\bfT$ and $P/W_0$ are natural $\SL_2(3)$-modules with respect to the action of $O^{3'}(\Out_\calF(P))$, where $W_0 := C_P(P/\bfW)$. We note that $P \cap R$ contains $\bfT$, meaning that $t_3$ does not centralize $P/(P \cap R) \cong S/R$. Similarly, $P \cap Q$ contains $W_0$, so $t_3$ cannot centralize $P/(P \cap Q) \cong S/Q$ either. We deduce that $|\Out_\calF(S)| = 2^3$. We shall further assume that $\langle t_1, t_2, t_3 \rangle$ is a Sylow $2$-subgroup of $\Aut_\calF(S)$.

        We now show that $\Aut_\calF(S)$ and $\Aut_{\calH_0}(S)$ are $\Aut(S)$-conjugate. Let $H$ be a Sylow $2$-subgroup of $\Aut(S)$ containing $H_0 := \langle t_1, t_2, t_3 \rangle$. It suffices to show that the choice of $H_0$ is unique in $H$. We recall that $t_1$ centralizes $Q/\bfU$ and $t_2$ centralizes $R/\bfU$. By Lemma \ref{lm:fi24-outfu}, we further know that $t_1$ centralizes $Z(\bfU)$. As $Z_2(\bfU)$ is self-centralizing, $t_1$ can be seen as a non-trivial element in $\Aut(Z_2(\bfU)/Z(\bfU)) \cong \GL_4(3)$. As $t_1|_{Z_2(\bfU)/Z(\bfU)}$ lies inside $\Aut_{O^{3'}(N_\calF(Q))}(Z_2(\bfU)/Z(\bfU)) \cong \Alt(5)$, which is isomorphic to a subgroup of $\SL_4(3)$. As such, the element $t_1|_{Z_2(\bfU)/Z(\bfU)}$ has determinant $1$, and $t_1$ centralizes $S/R$. This uniquely determines $t_1$. Similarly, $t_2$ is uniquely determined as an element of $\Aut(Z_2(\bfU)) \cong \GL_2(3)$ that acts as $-1$ on both $Z(\bfU)/Z(S)$ and $Z(S)$. Finally, $t_3$ is uniquely determined as an element that does not centralize $S/R$ and acts on $Z_2(\bfU)/Z(\bfU)$  with  determinant $1$. Thus, $H_0$ is uniquely determined inside $H$. This result also holds in $\Aut_{\calH_0}(S)$. We conclude that $\Aut_\calF(S)$ and $\Aut_{\calH_0}(S)$ are $\Aut(S)$-conjugate by Sylow's Theorems.
    \end{proof}
    
    Using all the information about we have gathered about the local structure of $\calF$, we will classify all fusion systems $\calF$ on $S$ such that $O_3(\calF) = 1$.
       
    \begin{theorem} \label{thm:fi24}
        Let $\calF$ be a saturated fusion system on $S$ with $O_3(\calF) = 1$. Then $\calF$ is realized by $\Fi_{24}'$ or $\Fi_{24}$.
    \end{theorem}
    \begin{proof}
        All code in this proof can be found in the files \texttt{fi24/uniqueness-r.g} and \texttt{fi24/uniqueness.m}. We recall that $G = \Fi_{24}'$. As $O_3(\calF) = 1$, we know by Proposition \ref{prp:fi24-trivcore} that $\calE(\calF) = \{P, Q, R\}$ and $O^{3'}(\Out_\calF(Q)) \cong \Alt(5)$. We set $\calG := \calF_S(H)$, where
        \[H := \begin{cases}
            G, & \textrm{if } |\Out_\calF(S)| = 2^3 \\
            \Aut(G), & \textrm{if } |\Out_\calF(S)| = 2^4.
        \end{cases}\]
         We know by Lemma \ref{lm:fi24-autfs} that there exists some $\alpha \in \Aut(S)$ such that $\Aut_\calF(S)^\alpha  = \Aut_\calG(S)$. Conjugating $\calG$ by $\alpha$, we shall assume that $\Aut_\calF(S) = \Aut_\calG(S)$.

        We now consider the choices for $\Aut_\calG(R)$. We compute using GAP that, for a fixed choice of $\Aut_{\Aut_\calG(S)}(R)$, there are precisely three choices for $\Out_\calG(R)$ in $\Out(R)$. Moreover, these choices are conjugate by some $\alpha_0 \in \Aut(S)$. As $\Aut_\calF(S) = \Aut_\calG(S)$, we know that
        \[\Aut_{\Aut_\calF(S)}(R) = \Aut_{\Aut_\calG(S)}(R).\]
        As such, we may arrange so that $\Aut_\calF(R) = \Aut_\calG(R)$. By the Model Theorem, there exists some $\beta \in \Aut(S)$ such that $N_{\calF^\beta}(R) = N_\calG(R)$. Since $\calF^{\beta} \cong \calF$, we may as well assume that $N_\calF(R) = N_\calG(R)$.

        Set $\calF_0 := N_\calF(\bfV_1)$ and $\calG_0 := N_\calG(\bfV_1)$. Since $N_\calF(R) = N_\calG(R)$, we find that that $\Aut_{N_{\calF_0}(\bfV_1)}(R) = \Aut_{N_\calF(\bfV_1)}(R)$. Indeed, we have that
        \[N_{\Aut_\calG(\bfV_1)}(\Aut_R(\bfV_1)) = N_{\Aut_\calF(\bfV_1)}(\Aut_R(\bfV_1)).\]
        Let $X := O^{3'}(\Out_\calF(\bfV_1))$ and $Y := O^{3'}(\Out_\calG(\bfV_1))$. By Proposition \ref{prp:fi24-outfv}, we know that $X$ and $Y$ are isomorphic to $\SO_7(3)$. Moreover, $X$ and $Y$ are $\Out(\bfV_1)$-conjugate by \cite[Table 8.35]{maximals}. We compute that
        \[N_{\Aut(\bfV_1)}(N_{\Aut_\calG(\bfV_1)}(\Aut_R(\bfV_1))) \leq N_{\Aut(\bfV_1)}(\Aut_\calF(\bfV_1)).\]
        As $R = C_S(Z_2(S))$ and $|Z_2(S)| = 3^2$, we further see that $\Aut_R(\bfV_1)$ is the unique maximal subgroup $M$ of $\Aut_S(\bfV_1)$ satisfying $|C_{\bfV_1}(M)| = 3^2$. This implies that $\Aut_R(\bfV_1)$ is weakly closed in $\Aut_S(\bfV_1)$ with respect to $\Aut(\bfV_1)$. By the Frattini Argument and Lemma \ref{lm:weak-closure-equality}, we deduce that $\Aut_{\calF}(\bfV_1) = \Aut_\calG(\bfV_1)$. We next compute that $H^1(\Aut_{N_{\calF_0}(R)}(\bfV_1); \bfV_1) = 0$. Given that we have the subsystems $\calF_0 \geq N_{\calF_0}(R) \leq \calG_0$, \cite[Proposition 2.11]{todd-modules} implies that $\calF_0 = \calG_0$. 
        
        We now set that $\calF_1 := \langle N_\calF(R), N_\calF(\bfV_1) \rangle_S$. By Lemmas \ref{lm:fi24-vnorm-p} and \ref{lm:fi24-vnorm-r}, we know that $O^{3'}(\Aut_\calF(P))$ and $O^{3'}(\Aut_\calF(R))$ are contained inside $N_\calF(\bfV_1)$. By Lemma \ref{lm:fi24-outfq}, we know that $O^{3'}(\Out_{N_\calF(\bfV_1)}(Q)) \cong \Alt(4)$. This is a maximal subgroup of $O^{3'}(\Out_\calF(Q)) \cong \Alt(5)$. By the extension axiom, there exists an involution $\gamma \in \Aut_\calF(S)$ such that $\gamma|_Q \in O^{3'}(\Aut_\calF(Q))$. Since $\gamma$ does not lie inside $O^{3'}(\Aut_{N_\calF(\bfV_1)}(Q))$, we infer that $O^{3'}(\Out_{\calF_1}(Q)) \cong \Alt(5)$. As $R$ is $\Aut_\calF(S)$-invariant, we further have that $\Aut_\calF(S) \subseteq N_\calF(R)$. By Alperin-Goldschmidt, we deduce that $\calF = \calF_1$. Using the same argument on $\calG$, we conclude that
        \[\calF = \langle N_\calF(R), N_\calF(\bfV_1) \rangle_S = \langle N_\calG(R), N_\calG(\bfV_1)\rangle_S = \calG.\]
    \end{proof}

    The following table summarizes the structure of the automizers in the two corefree fusion systems on $S$.
    \begin{table}[H]
        \centering
        \begin{tabular}{c|c|c|c|c}
            $\Out_\calF(P)$ & $\Out_\calF(Q)$ & $\Out_\calF(R)$ & $\Out_\calF(S)$ & $G$ \\
            \hline
            $2 \times \Sym(5)$ & $2 \times \GL_2(3)$ & $2 \times \GL_2(3)$ & $2 \times 2 \times 2$ & $\Fi_{24}'$ \\
            $2 \times 2 \times \Sym(5)$ & $2 \times 2 \times \GL_2(3)$ & $2 \times 2 \times \GL_2(3)$ & $2 \times 2 \times 2 \times 2$ & $\Fi_{24}$
        \end{tabular}
        \caption{Automizer structure for essential subgroups and $S$ in each corefree fusion system on a Sylow $3$-subgroup $S$ of $\Fi_{24}'$.}
    \end{table}
    
    \section{Fusion Systems on a Sylow $3$-subgroup of $\Mn$} \label{sec:m}
    In this section, we classify all corefree fusion systems on a Sylow $3$-subgroup of $\Mn$. From the ATLAS \cite{atlas}, we find that the following $3$-local subgroups are maximal in $\Mn$:
    \begin{align*}
        M_1 &= 3^{1+12}_+ \ldotp (2 \ldotp \Suz : 2) \\
        M_2 &= 3^{2+5+10} : (\Mt_{11} \times \GL_2(3)) \\
        M_3 &= 3^{3+2+6+6} : (\SL_3(3) \times \SDih(16)) \\
        M_4 &= 3^8 \ldotp (\SO_8^-(3) : 2)
    \end{align*}
    
    Fix a Sylow $3$-subgroup $S$ of $G := \Mn$. We can choose a $G$-conjugate of $M_i$ so that $S \in \Syl_3(M_i)$ for $1 \leq i \leq 4$. We can now construct further $3$-local subgroups of $G$:
    \begin{align*}
        N_1 = M_1 \cap M_3 &\cong 3^{1+12}_+ \ldotp 3^{2+4} : (\GL_2(3) \times \SDih(16)) \\
        N_2 = M_1 \cap M_2 &\cong 3^{1+12}_+ \ldotp 3^5 : (2^2 \times \Mt_{11}) \\
        N_3 = M_2 \cap M_3 &\cong 3^{2+5+10} : 3^2 : (\GL_2(3) \times \SDih(16)) 
    \end{align*}

    \begin{notation}
        Using the $3$-local subgroups of $G$ given above, we can define the following subgroups of $S$:
        \begin{enumerate}
            \item $\bfW := O_3(M_1)$;
            \item $\bfU := O_3(M_2)$;
            \item $\bfT := O_3(M_3)$;
            \item $\bfV_1 := O_3(M_4)$;
            \item $P := O_3(N_1)$;
            \item $Q := O_3(N_2)$; and
            \item $R := O_3(N_3)$.
        \end{enumerate}
    \end{notation}
    The group $\bfV_1$ has a second $N_G(S)$-conjugate that we shall label $\bfV_2$. We set $H_i := N_G(\bfV_i) \cong M_4$. We find that 
    \begin{align*}
        O^{3'}(\Out_{H_i}(P)) &= O^{3'}(\Out_G(P)) \cong \SL_2(3), \\ 
        \Alt(6) \cong O^{3'}(\Out_{H_i}(Q)) &< O^{3'}(\Out_G(Q)) \cong \Alt(6), \\
        O^{3'}(\Out_G(R)) &= O^{3'}(\Out_{H_i}(R)) \cong \SL_2(3). \\ 
    \end{align*}
    We note that $P/\bfV_i$, $Q/\bfV_i$ and $R/\bfV_i$ are precisely the unipotent radicals of the minimal parabolics in $H_i/\bfV_i \cong \SO^-_8(3)$.
    
    We now define the subgroup $X$ of $S$. Appealing to the ATLAS \cite{atlas}, we find that the group $O^{3'}(M_1/O_3(M_1)) \cong 2 \ldotp \Suz$ has a maximal subgroup isomorphic to $(\SL_2(9) \times 3^2) : Q_8$. Let $M_5 \leq M_1$ be the preimage of the maximal subgroup inside $\Mn$. We set $X := O_3(M_4)$. Then we have
    \[M_5 = N_{G}(X) \cong 3^{1+12}_+ \ldotp 3^2 : (\SL_2(9) : Q_8).\]
    We have $|X^S| = [S : N_S(X)] = 3^3$ and $|X^{N_G(S)}| = |X^{\Aut(S)}| = 2 \cdot 3^3$. Note further that $X$ is a subgroup of $P$, with $X/\bfW \cap Z(S/\bfW) = 1$. Moreover, no $\Aut(S)$-conjugate of $X$ contains either $\bfV_i$.

    These subgroups can also be characterized based on the structure of $S$:
    \begin{enumerate}
        \item $\bfW$ is the preimage of $J(S/Z(S))$ in $S$;
        \item $\bfT := C_S(Z_3(S))$;
        \item $\bfV_1$ and $\bfV_2$ are the two elementary abelian subgroups in $S$ of maximal order;
        \item $P = C_S(Z_3(S)/Z(S))$;
        \item $Q = C_S(Z_5(S)/Z_2(S))$;
        \item $R = C_S(Z_2(S))$; 
        \item $X$ is the unique subgroup of $S$ (up to $\Aut(S)$-conjugacy) containing $\bfW$ such that $[X : \bfW] = 3^2$, $N_S(N_S(X)) = P$, $[N_S(X), C_X(\gamma_3(X))] = \Phi(X)$, and every maximal subgroup $M$ of $X$ satisfies $\Phi(M) = \Phi(X)$; and
        \item $\bfU = Q \cap R$.
    \end{enumerate}
    Except for $\bfV_i$ and $X$, we see that every subgroup given above is characteristic in $S$. Moreover, the subgroups $\bfV_i$ are normal in $S$.

    \begin{lemma}
        Let $\calF := \calF_S(\Mn)$. Then $\calE(\calF) = \{P, Q, R, X^{\Aut(S)}\}$ and $\calF^{cr} = \calE(\calF) \cup \{\bfW, \bfT, \bfU, \bfV_1^{\Aut(S)}\}$.
    \end{lemma}
    \begin{proof}
        This follows from \cite{m-radicals}.
    \end{proof}

    We now let $\calF$ be a saturated fusion system on $S$. The online ATLAS \cite{online-atlas} has a permutation representation for $M_2$, which we can use to find an efficient way to represent $S$. We also construct the quotient $\Out_G(\bfW)$ in MAGMA by finding the relevant maximal subgroup in $O^{3'}(\Out(\bfW)) \cong \Sp_{12}(3)$.

    \begin{lemma}
        We have $\calE(\calF) \subseteq \{S, P, Q, R, X^{\Aut(S)}\}$.
    \end{lemma}
    \begin{proof}
        See Appendix \ref{sec:algorithm-appendix}.
    \end{proof}

    The code for Lemmas \ref{lm:m-w} to \ref{lm:m-u} can be found in the file \texttt{m/weak-closure.g}.
    \begin{lemma} \label{lm:m-w}
        The subgroup $\bfW$ is weakly closed in $\calF$. In particular, $\calE(N_\calF(\bfW))  = \calE(\calF) \cap \{P,Q,X^{\Aut(S)}\}$.
    \end{lemma}
    \begin{proof}
        We note that $P, Q$ and $X$ contain $\bfW$. Moreover, since $Z(\bfW) = Z(S)$, $\bfW$ is not a subgroup of $R$. Since $Z(S) = Z(P) = Z(Q) = Z(X)$ and $\bfW/Z(S) = J(S/Z(S))$, we deduce by Lemma \ref{lm:grp-thompson} that $\bfW/Z(Q) = J(Q/Z(Q))$, $\bfW/Z(P) = J(P/Z(P))$ and $\bfW/Z(X) = J(X/Z(X))$. We infer that $\bfW$ is normalized by $\Aut_\calF(P)$, $\Aut_\calF(Q)$, $\Aut_\calF(X)$ and $\Aut_\calF(S)$. By Alperin-Goldschmidt, we deduce that $\bfW$ is weakly closed in $\calF$. The second part follows from Lemma \ref{lm:essentials-in-normalizer}.
    \end{proof}

    \begin{lemma} \label{lm:m-t}
        The subgroup $\bfT$ is weakly closed in $\calF$. In particular, $\calE(N_\calF(\bfT)) = \calE(\calF) \cap \{P,R\}$.
    \end{lemma}
    \begin{proof}
        By construction, we have $\bfT \leq P \cap R$ with $S/\bfT \cong 3^{1+2}_+$. Moreover, $\bfT$ is not a subgroup of $Q$. Also, $\bfT$ is not contained inside any $\Aut(S)$-conjugate of $X$. We compute that $Z_3(S) = Z_2(P) = Z_2(R)$. This means that
        \[\bfT = C_S(Z_3(S)) = C_P(Z_2(P)) = C_R(Z_2(R)).\]
        As such, $\bfT$ is normalized by $\Aut_\calF(S)$, $\Aut_\calF(P)$ and $\Aut_\calF(R)$. The second part follows from Lemma \ref{lm:essentials-in-normalizer}.
    \end{proof}

    \begin{lemma} \label{lm:m-u}
        The subgroup $\bfU$ is weakly closed in $\calF$. In particular, $\calE(N_\calF(\bfU)) = \calE(\calF) \cap \{Q,R\}$.
    \end{lemma}
    \begin{proof}
        By construction, we have $\bfU = Q \cap R$ . Since $Q$ and $R$ are $\Aut_\calF(S)$-invariant, it follows that $\bfU$ is $\Aut_\calF(S)$-invariant. We also find that $\bfU = C_Q(Z_2(Q))$ and that $\bfU = C_R(Z_3(R)/Z(R))$. As such, $\bfU$ is normalized by $\Aut_\calF(A)$ for every $A \in \calE(\calF) \cup \{S\}$ containing $\bfU$. Therefore, $\bfU$ must be weakly closed in $\calF$. The second part follows from Lemma \ref{lm:essentials-in-normalizer}.
    \end{proof}

    The code for Lemmas \ref{lm:m-outfp} and \ref{lm:m-outfr} can be found in the file \texttt{m/automizer.g}.
    \begin{lemma} \label{lm:m-outfp}
        If $P \in \calE(\calF)$, then $P/\bfT$ and $Z_2(P)/Z(P)$ are both natural modules for $O^{3'}(\Out_\calF(P)) \cong \SL_2(3)$.
    \end{lemma}
    \begin{proof}
        Let $r \in O^{3'}(\Aut_\calF(P))$ be a $3'$-element that centralizes $Z_2(P)/Z(P)$. We compute that $|Z(P)| = 3$, $Z_2(P) = Z_3(S)$ and $Z_3(P) = Z_4(S)$. Then $\Aut_S(P)$ centralizes $Z_3(P)/Z_2(P)$. This means that $r \in O^{3'}(\Aut_\calF(R))$ centralizes $Z_3(P)/Z_2(P)$. Similarly, $r$ centralizes $Z(P)$ and $Z_2(P)/Z(P)$. By coprime action, we deduce that $r$ centralizes $Z_3(P)$. In that case, $[P, Z_3(P), r] = 1 = [Z_3(P), r, P]$. By the Three Subgroups Lemma, we infer that $[r, P, Z_3(P)] = 1$. In other words, $r$ centralizes $P/C_P(Z_3(P))$. We compute that $C_P(Z_3(P)) \leq \Phi(P)$, meaning that $r$ centralizes $P/\Phi(P)$. We conclude that $r = 1$. Thus, $O^{3'}(\Out_\calF(P))$ acts faithfully on $Z_2(P)/Z(P)$ of order $3^2$. This implies that $Z_2(P)/Z(P)$ is a natural module for $O^{3'}(\Out_\calF(P)) \cong \SL_2(3)$.

        Assume now that $r \in O^{3'}(\Aut_\calF(P))$ centralizes $P/\bfT$. Since $Z(\bfT) = Z_2(P)$, we deduce that $[r, P, Z_2(P)] = 1 = [P, Z_2(P), r]$. By the Three Subgroups Lemma, we infer that $[Z_2(P), r, P] = 1$. In other words, $r$ centralizes $Z_2(P)/Z(P)$. By the previous paragraph, we deduce that $r = 1$. Thus, $P/\bfT$ is a natural module for $O^{3'}(\Out_\calF(P)) \cong \SL_2(3)$.
    \end{proof}
    
    \begin{lemma} \label{lm:m-outfr}
        If $R \in \calE(\calF)$, then $R/\bfT$ and $Z(R)$ are natural modules for $O^{3'}(\Out_\calF(R)) \cong \SL_2(3)$.
    \end{lemma}
    \begin{proof}
        We compute that $[R : \Phi(R)] = 3^4$, $\Phi(R) = Z_6(R)$ and $[R : \bfT] = 3^2 = [\bfT : \Phi(R)]$. Moreover, we find that $[S, \bfT] \leq \Phi(R)$, so that $[O^{3'}(\Aut_\calF(R)), \bfT] \leq \Phi(R)$. As such, $O^{3'}(\Aut_\calF(R))$ must act faithfully on $R/\bfT$ of order $3^2$. Thus, $R/\bfT$ is a natural module for $O^{3'}(\Out_\calF(R)) \cong \SL_2(3)$.

        Let $r \in O^{3'}(\Aut_\calF(R))$ be a $3'$-element that centralizes $Z(R)$. We compute that $Z_2(R) = Z_3(S)$, $Z_3(R) = Z_4(S)$ and $Z_4(R) = Z_5(S)$. This implies that $r$ centralizes $Z_4(R)/Z_3(R)$, $Z_3(R)/Z_2(R)$ and $Z(R)$. By coprime action, we infer that $r$ centralizes $Z_4(R)$. As $Z_4(R)$ is self-centralizing in $S$, we deduce that $r = 1$. As such, $O^{3'}(\Out_\calF(R))$ must act faithfully on $Z(R)$ of order $3^2$. In other words, $Z(R)$ is a natural module for $O^{3'}(\Out_\calF(R)) \cong \SL_2(3)$.
    \end{proof}
    
    \begin{lemma} \label{lm:m-x-w-rad}
        If $X \in \calE(\calF)$, then $O_3(N_\calF(\bfW)) = \bfW$.
    \end{lemma}
    \begin{proof}
        Let $N := O_3(N_\calF(\bfW))$. As $\bfW \normalIn S$, we find that $N \normalIn S$. Moreover, since $X \in \calE(\calF)$, we find that $\bfW \leq N \leq X$ by Lemmas \ref{lm:weak-closure-to-normality} and \ref{lm:m-w}. We note that $X$ is not normal in $S$, with $[X : \bfW] = 3^2$. This implies that $[N : \bfW] \leq 3$. If $[N : \bfW] = 3$, then since $N \normalIn S$, we find that $N/\bfW$ is also normal in $S/\bfW$. As $|N/\bfW| = 3$, we are forced to have that $N/\bfW \leq Z(S/\bfW)$. But we recall that $X/\bfW \cap Z(S/\bfW) = 1$, which gives us a contradiction. We conclude that $N = \bfW$.
    \end{proof}
    
    \begin{proposition} \label{prp:m-outfw}
        Assume that $O_3(N_\calF(\bfW)) = \bfW$. Then $O^{3'}(\Out_\calF(\bfW)) \cong 2 \ldotp \Suz$ acts irreducibly on $\bfW/Z(S)$. Moreover, $\calE(N_\calF(\bfW)) = \{P, Q, X^{\Aut(S)}\}$, $O^{3'}(\Out_\calF(Q)) \cong \Mt_{11}$, $O^{3'}(\Out_\calF(X)) \cong \SL_2(9)$, and $|\Out_{O^{3'}(N_\calF(\bfW))}(S)| = 2^5$.
    \end{proposition}
    \begin{proof}
        The first part follows from Lemma \ref{lm:aut-sp12}. The second statement follows by analysing the structure of $\Out_\calF(\bfW)$ (see, for example, the file \texttt{m/automizer-w.m}.)
    \end{proof}

    \begin{lemma} \label{lm:m-outfx}
        If $X \in \calE(\calF)$, then $O^{3'}(\Out_\calF(X)) \cong \SL_2(9)$ acts irreducibly on $X/C_X(Z_2(X))$.
    \end{lemma}
    \begin{proof}
        All code used in this proof can be found in \texttt{m/automizer.g}. We know by Lemma \ref{lm:m-x-w-rad} that $O_3(N_\calF(\bfW)) = \bfW$. By Proposition \ref{prp:m-outfw}, we deduce that $O^{3'}(\Out_\calF(X)) \cong \SL_2(9)$. We compute that $[C_X(Z_2(X)), N_S(X)] = \Phi(X)$. This implies that $O^{3'}(\Out_\calF(X))$ acts trivially on $C_X(Z_2(X))/\Phi(X)$. By coprime action, we conclude that $O^{3'}(\Out_\calF(X))$ acts faithfully on $X/C_X(Z_2(X))$ of order $3^4$. Since $\SL_2(9)$ has no faithful $\GF(3)$-module of dimension less than $4$, we conclude that $O^{3'}(\Out_\calF(X))$ acts irreducibly on $X/C_X(Z_2(X))$.
    \end{proof}

    
    \begin{lemma} \label{lm:m-vnorm-r}
        If $R \in \calE(\calF)$, then both $\bfV_i$ are normalized by $O^{3'}(\Aut_\calF(R))$.
    \end{lemma}
    \begin{proof}
        All code used in this lemma can be found in the file \texttt{m/v-norm.g}. We compute that $Z_4(S) = Z_3(R)$ and $Z_3(S) = Z_2(R)$, with $[Z_3(R) : Z_2(R)] = 3^2$. As such, $O^{3'}(\Aut_\calF(R))$ centralizes $Z_3(R)/Z_2(R)$ of order $3^2$. In particular, the four intermediate subgroups of $Z_2(R)$ and $Z_3(R)$, which we label $C_i$ for $1 \leq i \leq 4$, are $O^{3'}(\Aut_\calF(R))$ invariant. We appeal to GAP to compute that, up to relabelling, the subgroups $J(C_R(Z_2(C_1)))$ and $J(C_R(Z_2(C_2)))$ are the two $\Aut(S)$-conjugates of $\bfV_1$. Thus, both $\bfV_i$ are normalized by $O^{3'}(\Aut_\calF(R))$.
    \end{proof}
    
    \begin{lemma} \label{lm:m-outfq}
        Assume that $\{P, Q\} \subseteq \calE(\calF)$. Then one of the following holds:
        \begin{enumerate}
            \item $O_3(N_\calF(\bfW)) = \bfW$, and $O^{3'}(\Out_\calF(Q)) \cong \Mt_{11}$; or
            \item $O_3(N_\calF(\bfW)) \neq \bfW$, in which case some $\bfV_i$ is normal in $N_\calF(\bfW)$.
        \end{enumerate}
    \end{lemma}
    \begin{proof}
        All code used in this proof can be found in the file \texttt{m/automizer-q.g}. We know that $[\bfW : \Phi(Q)] = 3$, so $O^{3'}(\Aut_\calF(Q))$ centralizes $\bfW/\Phi(Q)$. By coprime action, we deduce that $O^{3'}(\Out_\calF(Q))$ acts faithfully on $Q/\bfW$ of order $3^5$. In particular, $O^{3'}(\Out_\calF(Q))$ is a subgroup of $\SL_5(3)$. Since $\Out_S(Q)$ is elementary abelian of order $3^2$, we appeal to \cite[Tables 8.18 and 8.19]{maximals} to see that there are three possible choices for $O^{3'}(\Out_\calF(Q))$ -- $\Alt(6)$, $\SL_2(9)$ or $\Mt_{11}$. 
        
        We know that $\Mt_{11}$ has no faithful $\GF(3)$-module of dimension less than $5$. As such, if $O^{3'}(\Out_\calF(Q)) \cong \Mt_{11}$, then $O^{3'}(\Out_\calF(Q))$ acts irreducibly on $Q/\bfW$. Since $P \in \calE(\calF)$, we deduce that $O_3(N_\calF(\bfW)) = \bfW$.
    
        Assume now that $O^{3'}(\Out_\calF(Q))$ is isomorphic to either $\Alt(6)$ or $\SL_2(9)$. Let $A := O_3(N_\calF(\bfW)) \leq P \cap Q$. If $A = P \cap Q$, then $O^{3'}(\Out_\calF(Q))$ normalizes $Z_2(A) = Z_2(S)$, which contradicts Lemma \ref{lm:m-outfp}. Combining with Proposition \ref{prp:m-outfw}, we deduce that $3^2 \leq [Q : A] \leq 3^4$ and $3 \leq [A : \bfW] \leq 3^3$. Since $\Alt(6)$ or $\SL_2(9)$ has no faithful $\GF(3)$-module of dimension less than $4$, we infer that $O^{3'}(\Out_\calF(Q))$ acts irreducibly on $Q/A$ of order $3^4$ and centralizes $A/\bfW$ of order $3$. As $A/\bfW$ is normal in $S/\bfW$ and has order $3$, we infer that $A/\bfW \leq Z(S/\bfW)$. Then there are precisely four choices for $A$. We show that only two of the four choices are valid. To see this, consider the subgroup $\hat{C} := \langle A, C_Q(\gamma_3(A)) \rangle$. As $A$ is invariant under $O^{3'}(\Aut_\calF(Q))$, $\hat{C}$ must also be invariant under $O^{3'}(\Aut_\calF(Q))$. We compute that $\hat{C}$ properly contains $A$ in each case, so it must equal $Q$. This allows us to deduce that there are precisely two choices for $A$, which we label $A_1$ and $A_2$. We observe that $A_1$ and $A_2$ are $\Aut(S)$-conjugate. Then, up to relabelling, we have that $\bfV_i = C_{A_i}(Z_2(A_i))$. In particular, as some $A_i$ is normal in $N_\calF(\bfW)$, we deduce that some $\bfV_i$ is also normal in $N_\calF(\bfW)$.
    \end{proof}

    \begin{proposition} \label{prp:m-trivcore}
        Let $\calF$ be a saturated fusion system on $S$. Then either $\calF$ is constrained or we have $O_3(\calF) = 1$. Moreover, $O_3(\calF) = 1$ if and only if $\calE(\calF) = \{P, Q, R, X^{\Aut(S)}\}$ and $O^{3'}(\Out_\calF(Q)) \cong \Mt_{11}$.
    \end{proposition}
    \begin{proof}
        If $R \not\in \calE(\calF)$, then we know by Lemma \ref{lm:m-w} that $\bfW \normalIn \calF$. If $Q \not\in \calE(\calF)$, then Lemma \ref{lm:m-t} tells us that $\bfT \normalIn \calF$. If $P \not\in \calE(\calF)$, then we apply Lemma \ref{lm:m-u} to find that $\bfU \normalIn \calF$. If $O^{3'}(\Out_\calF(Q)) \not\cong \Mt_{11}$, then we deduce by Lemma \ref{lm:m-outfq} that some $\bfV_i$ is normal in $N_\calF(\bfW)$. Appealing to Lemma \ref{lm:m-vnorm-r}, we further see that $\bfV_i$ is normalized by $O^{3'}(\Aut_\calF(R))$. By the Frattini Argument and Lemma \ref{lm:autfe-decomposition}, we have that $\calF = \langle N_\calF(\bfW), O^{3'}(\Aut_\calF(R)) \rangle$. This implies that $\bfV_i \normalIn \calF$. Instead, if $O^{3'}(\Out_\calF(Q)) \cong \Mt_{11}$, then Lemma \ref{lm:m-outfq} tells us that $O_3(N_\calF(\bfW)) = \bfW$. Moreover, by Proposition \ref{prp:m-outfw}, we deduce that $X \in \calE(\calF)$. As $\bfW, \bfT, \bfU$ and $\bfV_i$ are self-centralizing in $S$, we infer that either $\calF$ is constrained or $\calE(\calF) = \{P, Q, R, X^{\Aut(S)}\}$ with $O^{3'}(\Out_\calF(Q)) \cong \Mt_{11}$. If $O_3(\calF) = 1$, then $\calF$ cannot be constrained. Thus, if $O_3(\calF) = 1$, then we must have $\calE(\calF) = \{P, Q, R, X^{\Aut(S)}\}$, with $O^{3'}(\Out_\calF(Q)) \cong \Mt_{11}$.
        
        Assume now that $\calE(\calF) = \{P, Q, R, X^{\Aut(S)}\}$ and $O^{3'}(\Out_\calF(Q)) \cong \Mt_{11}$. Then by Lemma \ref{lm:m-outfq}, we deduce that $O_3(N_\calF(\bfW)) = \bfW$. By Proposition \ref{prp:m-outfw}, we have that $O^{3'}(\Out_\calF(\bfW)) \cong 2 \ldotp \Suz$ acts irreducibly on $\bfW/Z(S)$. Also, we know by Lemma \ref{lm:m-outfr} that $O^{3'}(\Out_\calF(R))$ acts irreducibly on $Z(R) = Z_2(S)$. As $\bfW \nleq R$, we conclude that $O_3(\calF) = 1$.
    \end{proof}
    
    \begin{lemma} \label{lm:m-autfs}
        If $O_3(\calF) = 1$, then $|\Aut_\calF(S)|_{3'} = 2^6$.
    \end{lemma}
    \begin{proof}
        Since $O_3(\calF) = 1$, we know by Proposition \ref{prp:m-trivcore} that $\calE(\calF) = \{P, Q, R, X^S\}$ and $O^{3'}(\Out_\calF(Q)) \cong \Mt_{11}$. In that case, Proposition \ref{prp:m-outfw} tells us that $O^{3'}(\Out_\calF(\bfW)) \cong 2 \ldotp \Suz$, with $|\Aut_{O^{3'}(N_\calF(\bfW))}(S)|_{3'} = 2^5$. Also, Lemma \ref{lm:m-outfr} tells us that $O^{3'}(\Aut_\calF(R))$ acts irreducibly on $Z(R)$. In that case, the extension axiom tells us that there exists an involution $t \in \Aut_\calF(S)$ such that $t|_{R} \in O^{3'}(\Aut_\calF(R))$ and $t$ doesn't centralize $Z(S)$. Since $O^{3'}(\Out_\calF(\bfW))$ centralizes $Z(S)$, we infer that $t \not\in \Aut_{O^{3'}(N_\calF(\bfW))}(S)$. We deduce that $|\Aut_\calF(S)|_{3'} \geq 2^6$. We compute that $|\Aut(S)|_{3'} = 2^6$, and so the result follows.
    \end{proof}
    
    Using all the information we have gathered in this section, we will classify all fusion systems $\calF$ on $S$ such that $O_3(\calF) = 1$.
    
    \begin{theorem} \label{thm:m}
        Let $\calF$ be a saturated fusion system on $S$ with $O_3(\calF) = 1$. Then $\calF$ is realized by $\Mn$.
    \end{theorem}
    \begin{proof}
        All code in this proof can be found in the files \texttt{m/uniqueness.m} and \texttt{m/uniqueness.g}. We set $\calG := \calF_S(\Mn)$. By Lemma \ref{lm:m-autfs}, we find that $\Aut_\calF(S)$ and $\Aut_\calG(S)$ are both Sylow $2$-subgroups of $\Aut(S)$. By the Sylow's Theorems, we infer that $\Aut_\calG(S)^\alpha  = \Aut_\calF(S)$ for some $\alpha \in \Aut(S)$. We may replace $\calG$ by the isomorphic fusion system $\calG^\alpha$, and apply the Model Theorem to deduce that $N_\calF(S) = N_\calG(S)$. We now show that $\calF = \calG$.

        Now, we note that $\Aut_\calF(R) \geq \Aut_S(R) \leq \Aut_\calG(R)$ with
        \[N_{\Aut(R)}(N_{\Aut_\calG(R)}(\Aut_S(R))) \leq N_{\Aut(R)}(\Aut_\calG(R)).\]
        We compute that $\Out(R)$ has a unique conjugacy class of subgroups isomorphic to $\Out_\calF(R) \cong \GL_2(3) \times \SDih(16)$. Since $\Aut_\calF(S) = \Aut_\calG(S)$, we have that
        \[N_{\Aut_\calF(R)}(\Aut_S(R)) = N_{\Aut_\calG(R)}(\Aut_S(R)).\]
        Then Lemma \ref{lm:weak-closure-equality} implies that $\Aut_\calF(R) = \Aut_\calG(R)$. 

        We next consider $N_\calF(\bfW)$. We note that $\Aut_S(\bfW)$ is a Sylow $3$-subgroup of both $\Aut_\calF(\bfW)$ and $\Aut_\calG(\bfW)$. Let $X := O^{3'}(\Out_\calF(\bfW))$ and $Y := O^{3'}(\Out_\calG(\bfW))$. By Proposition \ref{prp:m-outfw}, we know that $X$ and $Y$ are isomorphic to $2 \ldotp \Suz$. Moreover, $X$ and $Y$ are $\Out(\bfW)$-conjugate by \cite[Table 8.80]{maximals}. We further compute that
        \[N_{\Out(\bfW)}(N_{\Out_\calG(\bfW)}(\Out_S(\bfW))) = N_{\Out_\calG(\bfW)}(\Out_S(\bfW)) \leq N_{\Out(\bfW)}(\Out_\calG(\bfW)).\]
        By the Frattini Argument and Lemma \ref{lm:weak-closure-equality}, we infer that $\Aut_\calF(\bfW) = \Aut_\calG(\bfW)$. Moreover, we compute that $H^1(\Aut_{N_\calG(S)}(\bfW); Z(\bfW)) = 0$. Since $N_\calF(\bfW) \geq N_\calF(S) \leq N_\calG(\bfW)$, we apply \cite[Proposition 2.11]{todd-modules} to deduce that $N_\calF(\bfW) = N_\calG(\bfW)$. 
        
        We have now found that $\Aut_\calF(R) = \Aut_\calG(R)$ and $N_\calF(\bfW) = N_\calG(\bfW)$. Since $\Aut_\calF(S)$, $\Aut_\calF(X)$, $\Aut_\calF(P)$ and $\Aut_\calF(Q)$ normalize $\bfW$, we conclude by Alperin-Goldschmidt that 
        \[\calF = \langle N_\calF(\bfW), \Aut_\calF(R) \rangle_S = \langle N_\calG(\bfW), \Aut_\calG(R) \rangle_S = \calG.\]
    \end{proof}
    
    The following table summarizes the structure of the automizers in the unique corefree fusion system on $S$.
    \begin{table}[H]
        \centering
        \begin{tabular}{c|c|c|c|c|c}
            $\Out_\calF(P)$ & $\Out_\calF(Q)$ & $\Out_\calF(R)$ & $\Out_\calF(X)$ & $\Out_\calF(S)$ & $G$ \\
            \hline
            $\GL_2(3) \times \SDih(16)$ & $2^2 \times \Mt_{11}$ & $\GL_2(3) \times \SDih(16)$ & $\SL_2(9) : Q_8$ & $2^2 \times \SDih(16)$ & $\Mn$
        \end{tabular}
        \caption{Automizer structure for essential subgroups and $S$ in the unique corefree fusion system on a Sylow $3$-subgroup $S$ of $\Mn$.}
    \end{table}

    \section*{Acknowledgments}
    This work took place during the author's PhD at the University of Manchester under the supervision of Prof. Charles Eaton and Dr. Martin van Beek. The author acknowledges the support received from EPSRC (EP/W524347/1) during this period. 
    
    The author thanks Dr. Martin van Beek for his supervision and constant support. The author also thanks Prof. Charles Eaton for proofreading the paper.

    \appendix

    \section{Performance of the algorithm} \label{sec:algorithm-appendix}
    In this section, we compare the performance of Algorithm \ref{algorithm} to the previous algorithm described in \cite[Appendix A]{paper:paper-1}. The algorithm presented in that paper is itself an implementation in GAP of the Parker-Semeraro MAGMA package \cite{ps-magma}. We refer the reader to \cite[Appendix A]{paper:paper-1} for the definition of proto-essential subgroups and the specific tests involved. The file \texttt{find-proto-essentials.g} has an implementation of Algorithms \ref{algorithm} and \ref{alg:all-proto} in GAP that we use here.
    
    The table below compares how many subgroups the two algorithms consider when finding all the proto-essential subgroups on a Sylow $3$-subgroup of $\Fi_{22}$ (of order $3^9$) and of $\Fi_{23}$ (of order $3^{13}$). The count of subgroups is listed up to $\Aut(S)$-conjugacy.
    \begin{table}[H]
        \centering
        \resizebox{\textwidth}{!}{\begin{tabular}{c|c|c|c|c}
            \multirow{2}{*}{$G$} & \multicolumn{2}{|c|}{$\Fi_{22}$} &  \multicolumn{2}{|c}{$\Fi_{23}$} \\
            \cline{2-5}
            & Previous Algorithm & New Algorithm & Previous Algorithm & New Algorithm \\
            \hline
            Subgroups Considered & $986$ & $38$ & $177 \ 864$ & $158$  \\
            Centric Test & $233$ & $38$ & $11 \ 894$ & $158$ \\
            Rank Test & $128$ & $30$ & $2 \ 620$ & $124$\\
            Frattini Test & $50$ & $17$ & $324$ & $48$ \\
            Radical Test & $7$ & $5$ & $6$ & $6$ \\
            Lift Test & $3$ & $3$ & $6$ & $6$
        \end{tabular}}
        \caption{Comparison of the number of subgroups passing a particular test of proto-essential subgroups in the two algorithms for a Sylow $3$-subgroup of $\Fi_{22}$ and $\Fi_{23}$.}
    \end{table}
    \noindent We recall in Section \ref{sec:alg}, we demonstrated that all the subgroups considered by the new algorithm pass the centric test. It is clear that the new algorithm is a major improvement to the previous algorithm, especially as the size of the group grows. In practice, the new algorithm is much faster than the previous one as well.
    
    The following table lists the number of subgroups passing each test as we compute the proto-essential subgroups in a Sylow $3$-subgroup of $\Fi_{24}'$ and $\Mn$, starting with subgroups found as a result of running Algorithm \ref{algorithm}.
    \begin{table}[H]
        \centering
        \begin{tabular}{c|c|c}
            $G$ & $\Fi_{24}'$ & $\Mn$ \\
            \hline
            Subgroups Considered & $572$ & $787$ \\
            Rank Test & $332$ & $285$ \\
            Frattini Test & $76$ & $62$ \\
            Radical Test & $4$ & $\leq 13$ \\
            Lift Test & $3$ & $\leq 10$ \\
            \hline
            Union of $\calE(\calF)$ & $3$ & $4$ 
        \end{tabular}
        \caption{Number of subgroups passing a particular test of proto-essential subgroups for a Sylow $3$-subgroup of $\Fi_{24}'$ and $\Mn$.}
    \end{table}
    \noindent Because the radical and lift tests are rather expensive, we do not run those tests on the subgroups that are known to be essential in some fusion system on $S$. As a Sylow $3$-subgroup of $\Mn$ has order $3^{20}$, we are unable to compute either $\Aut(E)$ or $\Aut(N_S(E))$ for six subgroups $E$ of $S$ that we consider. This excludes subgroups ruled out by the test mentioned in the remark after Corollary \ref{cor:alg}. We rule these subgroups out manually using GAP, which is given in \texttt{m/essentials.g}. We provide two examples here.
    \begin{enumerate}
        

        \item There is $A_1 \leq S$ such that $|A_1| = 3^{15}$ with $|N_S(A_1)| = 3^{17}$, and
        \[Z(A_1) = Z_3(S) = Z(N_S(A_1))\]
        of order $3^3$. Moreover, $\Phi(A_1) = Z_2(A_1)$ is self-centralizing in $S$ and has order $3^3$. By coprime action, we see that any $3'$-element $r \in O^{3'}(\Aut_\calF(A_1))$ centralizes $Z(A_1)$. As $\Phi(A_1) = Z_2(A_1)$ is self-centralizing in $S$, we infer that $O^{3'}(\Out_\calF(A_1))$ acts faithfully on $\Phi(A_1)/Z(A_1)$ of order $3^4$. We compute that $\Out_S(A_1)$ is elementary abelian of order $3^2$, and that $|C_{\Phi(A_1)/Z(A_1)}(S)| = 3^2$. We deduce that $\Phi(A_1)/Z(A_1)$ is a natural module for $O^{3'}(\Out_\calF(A_1)) \cong \SL_2(9)$.

        Let $\tau \in O^{3'}(\Aut_\calF(A))$ be an involution. We first assume that $\tau$ inverts $A_1/\Phi(A_1)$. We have found that $\Phi(A_1)/Z(A_1)$ is a natural $\SL_2(9)$-module. Thus, $\tau$ inverts $\Phi(A_1)/Z(A_1)$ as well. This implies that $\tau$ inverts $A_1/Z(A_1)$. but then $A_1/Z(A_1)$ is abelian, which is a contradiction. Let $V := A_1/\Phi(A_1)$, which has order $3^8$. By coprime action, we know that
        \[V = [V, \tau] \times C_V(\tau).\]
        We have found that $[V, \tau] \neq V$. As $|V| = 3^8$, and the minimal faithful dimension of an $\SL_2(9)$-module over $\GF(3)$ is $4$, we deduce that $[V, \tau]$ is a natural $\SL_2(9)$-module. On the other hand, $\tau$ centralizes $C_V(\tau)$. Thus, $C_V(\tau)$ is either a natural $\O_4^-(3)$-module, where $\O_4^-(3) \cong \L_2(9) \cong \Alt(6)$, or $C_V(\tau) = C_V(O^{3'}(\Out_\calF(A_1)))$. We compute that $[[N_S(A_1), A_1] : \Phi(A_1)] = 3^4$. But this does not match the module structure in either case, and we have a contradiction.

        \item There is $A_2 \leq S$ such that $|A_2| = 3^{17}$ with $N_S(A_2)$ of order $3^{19}$, and $Q' \leq \bfW \leq A_2$. We recall that $\bfW$ is an extraspecial subgroup of $S$ of order $3^{13}$ of exponent $3$. Moreover, $X := \langle \Phi(A_2), Z_4(A_2) \rangle$ is such that $[N_S(A_2), A_2] \leq X$ and $[X : \Phi(A_2)] = 3$. Thus, $\Aut_S(A_2)$, and so $O^{3'}(\Aut_\calF(A_2))$, centralizes the quotients $A_2/X$ and $X/\Phi(A_2)$. Therefore, any $3'$-element $r \in O^{3'}(\Aut_\calF(A_2))$ centralizes $A_2/\Phi(A_2)$ by coprime action, so that $r = 1$. As such, $O^{3'}(\Aut_\calF(A_2)) = \Aut_S(A_2)$, contradicting that $O_3(\Out_\calF(A_2)) = 1$.
    \end{enumerate}
    
    For Sylow $3$-subgroups of $\Fi_{24}'$ and $\Mn$, we do not need to make use of Algorithm \ref{alg:all-proto} to find all the proto-essentials. The essential subgroups found in the $3$-fusion category of $\Fi_{24}'$ are all found after running Algorithm \ref{algorithm}. Moreover, the three subgroups are all maximal in $S$, meaning that these must precisely be the proto-essentials (cf. the remark after Corollary \ref{cor:main-alg}). In the $3$-fusion category of $\Mn$, there are three essential subgroups that are not contained in any other essential subgroup. The final essential subgroup $X$ was uniquely determined up to $\Aut(S)$-conjugacy in Section \ref{sec:m}, so there cannot be any further proto-essential subgroups.

    We end the section by considering some consequences of Corollary \ref{cor:opf}. The corollary gives us an incredibly efficient way of checking whether a $p$-group can support corefree fusion systems. For relatively small groups, this typically involves checking less than $5$ groups, as can be seen below.
    \begin{table}[H]
        \centering
        \begin{tabular}{c|c|c|c}
            \multirow{2}{*}{$p^n$} & \multicolumn{3}{|c}{Subgroups considered} \\
            & Mean & Median & Maximum \\
            \hline
            $2^4$ & $0.7$ & $1$ & $1$ \\
            $p^4$ ($p$ odd) & $1.8$ & $2$ & $2$ \\
            \hline
            $2^5$ & $1.2$ & $1$ & $3$ \\
            $3^5$ & $1.8$ & $2$ & $5$ \\
            $5^5$ & $2.0$ & $2$ & $4$ \\
            $7^5$ & $1.9$ & $2$ & $4$ \\
            $11^5$ & $1.8$ & $1$ & $4$ \\
            \hline
            $2^6$ & $1.4$ & $1$ & $4$ \\
            $3^6$ & $2.5$ & $2$ & $8$ \\
            $5^6$ & $2.7$ & $2$ & $9$ \\
            $7^6$ & $2.7$ & $2$ & $10$ \\
            \hline
            $2^7$ & $1.9$ & $2$ & $7$ \\
            $3^7$ & $3.9$ & $4$ & $19$ \\
            \hline
            $2^8$ & $2.3$ & $2$ & $12$
        \end{tabular}
        \caption{Aggregate number of subgroups for groups order $p^n$ (up to $\Aut(S)$-conjugacy) for which the proto-essentials test needs to be called in order to check whether the group can support corefree fusion systems.}
        \label{tbl:csx}
    \end{table}
    \noindent Groups of class $2$ are not included -- all $p$-groups of class $2$ that support corefree fusion systems are classified in \cite{paper:class2}. We see that, on average, one needs to only check whether an incredibly small number of subgroups can be proto-essential to conclude whether it can support corefree fusion systems. This is a major improvement to having to first find all proto-essential subgroups, and then construct possible corefree fusion systems to find such $p$-groups.
    
    Note that we do not have a classification of corefree fusion systems on a $p$-group of order $7^6$. Assuming we want to classify all corefree fusion systems on groups of this order, we can:
    \begin{enumerate}
        \item check whether the $p$-group $S$ can support corefree fusion systems using Corollary \ref{cor:opf};
        \item if so, use Algorithm \ref{alg:all-proto} to find all the proto-essential subgroups of $S$; and
        \item construct fusion systems $\calF$ on $S$ and see if they are corefree.
    \end{enumerate}
    Currently, we have only implemented steps (1) and (2) in GAP.
    
    Being more optimistic, one could use this method to find all corefree fusion systems on groups of order $2^{10}$. We note that the smallest Benson--Solomon system is supported on a group of order $2^{10}$, making this an important project in the theory of fusion systems. Looking at $2$-groups in Table \ref{tbl:csx}, we predict that for any group $S$ of order $2^{10}$, one only needs to check whether around $4$ subgroups of $S$ can be proto-essential to see if the group can support corefree fusion systems. Combining with the tests given in \cite[Proposition 2.2]{aov}, we believe that the problem is now tractable (see also \cite[Chapter 11]{thesis:2-10} for an attempt on this project).

    \bibliographystyle{alpha}
    \bibliography{nat}

\end{document}